\documentclass[a4paper,reqno]{amsart}

\usepackage{amssymb,graphicx,xcolor}
\usepackage{amssymb,bbm,enumerate}
\usepackage{esint}
\usepackage{colortbl}
\usepackage{xcolor}

\numberwithin{equation}{section}\newtheorem{theorem}{Theorem}[section]
\newtheorem{lemma}[theorem]{Lemma}

\newtheorem{proposition}[theorem]{Proposition}\theoremstyle{remark}
\newtheorem{remark}{Remark}[section]
\theoremstyle{definition}
\newtheorem{definition}[theorem]{Definition}

\newtheorem{properties}{Property}[section]

\newcommand{\dist}{\mathrm{dist}}
\newcommand{\supp}{\mathrm{supp\,}}
\newcommand{\R}{\mathbb{R}}
\newcommand{\Rn}{\mathbb{R}^{n}}

\title[Average decay of the Fourier transform of measures]
{Average decay of the Fourier transform of measures with applications}
\date{\today}    %%% ''\date{}'' to omit date

\author{Renato Luc\`a}
\author{Keith M. Rogers}
\address{Instituto de Ciencias Matem\'aticas CSIC-UAM-UC3M-UCM, Madrid, 28049, Spain.}
\email{renato.luca@icmat.es, keith.rogers@icmat.es}

\thanks{
Mathematics Subject Classification. Primary 42B37; Secondary  28A75}

\keywords{Hausdorff measure and dimension, Fourier transform}

\begin{document}
\begin{abstract} We consider spherical averages of the Fourier transform of fractal measures and improve both the upper and lower bounds on the rate of decay. Maximal estimates 
with respect to fractal measures are deduced
for the Schr\"odinger and wave equations. This refines the almost everywhere convergence of the solution to its initial datum as time tends to zero. A consequence is that the solution to the wave equation cannot diverge on a $(d-1)$-dimensional manifold if the data belongs to the energy space $\dot{H}^1(\mathbb{R}^d)\times L^2(\mathbb{R}^d)$.  
\end{abstract}

\maketitle

\section{Introduction}

Consider the  Schr\"odinger equation, $i \partial_{t} u +\Delta u=0$, on $\R^d=\R^{n+1}$, with initial data 
$u(\cdot, 0) = u_{0}$ in $H^{s}$ defined by
$$
H^{s} = \left\{ 
G_{s} * f \ : \ f \in L^{2}(\Rn)
\right\}.
$$
Here $G_{s}$ is the Bessel kernel defined as usual by $\widehat{G}_{s} = (1 + |\cdot|^{2})^{- s/2}$, where $\,\widehat{\ }\,$ is the Fourier transform.
In \cite{Carl}, Carleson 
considered the problem of identifying 
the exponents $s > 0$ for which
\begin{equation}\label{CarlesonProblem}
\lim_{t\to 0} u(x , t) = u_{0}(x), \qquad \text{a.e.} \quad x \in \Rn, \qquad \forall \  u_0\in H^s,
\end{equation}
and proved that this is true as long as
 $s \geq 1/4$ in the one-dimensional case. Dahlberg and Kenig \cite{DahlKenig} then showed that \eqref{CarlesonProblem} does not hold if $s<1/4$. 
The higher dimensional case has since been studied by many authors; see for example
\cite{Cow, Carb, Sjolin,Vega,Bou3,MVV2,TV,T}. The best known positive result to date, that (\ref{CarlesonProblem}) holds if 
$s > 1/2 - 1/(4n)$, is due to Lee \cite{Lee} when $n=2$ and Bourgain \cite{Bou2} when $n\ge 3$. Bourgain also showed that $s\ge 1/2-1/n$ is necessary for \eqref{CarlesonProblem} to hold.
Together we see that \eqref{CarlesonProblem} holds uniformly with respect to $n$ if and only if $s\ge 1/2$.

A natural refinement of the problem 
is to bound the size of the divergence sets 
\begin{equation*}
\mathcal{D}(u_{0}) := \Big\{\, 
x \in \Rn\ : \lim_{t\to 0} u(x, t) \neq u_{0}(x)
\,\Big\},
\end{equation*}
 and in particular we consider 
\begin{equation*}
\alpha_{n}(s) := 
\sup_{u_{0}\in H^{s}} \dim_{H}\big( \mathcal{D}(u_{0}) \big),
\end{equation*}
where $\dim_{H}$ denotes the Hausdorff dimension. A completely satisfactory theory has already been developed in the one-dimensional case; see \cite{BBCR}, \cite{BR}, or \cite{CL}. Indeed
$$
\alpha_{n}(s) \leq\begin{array}{lcl}
n-2s,  & \mbox{if} & 
 \frac{n}{4}  \le s \le \frac{n}{2},  \\
\end{array}
$$
and this bound is sharp in the sense that initial data in $H^s$ can be singular on $\alpha$-dimensional sets when $\alpha<n-2s$; see \cite{Z}. On the other hand, the solution is continuous (and so $\alpha_{n}(s)=0$) when $s>n/2$, and the example of Dahlberg and Kenig tells us that $\alpha_{n}(s)=n$ when $s<1/4$.  Noting that altogether, when $n=1$, we have covered the whole range, we see that  $\alpha_1$ is known and it is discontinuous at $s=1/4$.   
These results and a more gentle introduction to the problem can also be found in  \cite[Chapter 17]{M}. 

Here we improve the best known upper 
bounds for $\alpha_{n}(s)$ in the remaining range of interest, when $s< n/4$, in higher dimensions.  In particular, we prove the following theorem that refines the almost everywhere convergence due to Bourgain and Lee. At the same time, we improve the bounds $\alpha_n(s)\le n+1-2s$ due to Sj\"ogren and Sj\"olin~\cite{Sjolin2} and  $\alpha_{n}(s)\le \frac{n+3}{n + 1}\big(n-2s)$ due to Barcel\'o, Bennett, Carbery and the second author \cite{BBCR}.

\begin{theorem}\label{bil} Let $n\ge 2$. Then
$$
\alpha_{n}(s) \leq \left\{
\begin{array}{lcl}
n + 1 - \big( 2+\frac{2}{2n-1} \big) s,  &  & 
\frac{1}{2} - \frac{1}{4n} < s \le1 - \frac{3}{2(n+1)},  \\
 &  & \\
n +1-\frac{1}{n+1}- 2s,  &  &  
1 - \frac{3}{2(n+1)} \leq s < \frac{n}{4}.
\end{array} \right.
$$
\end{theorem}

This will be a consequence of a maximal estimate (see Theorem~\ref{OurBouThm}) that holds uniformly with respect to fractal measures in the following class. To avoid repetition, we include positivity and a support condition inside the definition of \lq $\alpha$-dimensional'.

\begin{definition}
Let $0 < \alpha \leq d$.  We say that $\mu$ is (at least) $\alpha$-dimensional if it is a positive Borel measure, supported in the unit ball $B(0,1)$, that satisfies
$$
c_{\alpha}(\mu) := \sup_{\substack{x \in \R^d \\ r > 0}}
\frac{\mu(B(x,r))}{r^{\alpha}} < \infty.
$$
\end{definition}

The Fourier transform of such a measure need not decay in every direction (for example the Fourier transform of a piece of the surface measure on a hyperplane does not decay in the normal direction), however it must decay on average. As the class contains measures that are supported on $\alpha$-dimensional sets, the uncertainty principle suggests that there should be less decay for smaller values of~$\alpha$.  Let $\beta_d(\alpha)$
denote the supremum of the numbers $\beta$ for which\footnote{We write $A \lesssim B$ if $ A\leq C B$ for some constant $C > 0$ that only depends on the dimension $d$ and/or a small parameter $\varepsilon$, in this case $\varepsilon=\beta_d(\alpha)-\beta$. If the constant depends on anything else, say a power of~$N$, we  write $A \lesssim_N B$. We also write $A \simeq B$ if  $ A \lesssim B$ and $B \lesssim A$.}
\begin{equation}\label{dz}
\|\widehat{\mu}(R\,\cdot\,)\|^2_{L^2(\mathbb{S}^{d-1})}\lesssim
c_{\alpha}(\mu)\|\mu\|R^{-\beta}\end{equation} whenever 
$R> 1$ and $\mu$ is $\alpha$-dimensional. The problem of identifying the precise value of $\beta_d(\alpha)$ was proposed by Mattila; see for example \cite[pp. 42]{M4} or \cite[Chapter 15]{M}. 
In two dimensions,  the sharp decay rates are now known;
$$
\beta_2(\alpha)=\left\lbrace
\begin{array}{lll}
\alpha,& \alpha\in(0,1/2],& \\
&& \text{(Mattila~\cite{M0})}\\
1/2,& \alpha\in[1/2,1],& \\
&&\\
\alpha/2,& \alpha\in[1,2],& \text{(Wolff~\cite{W})}.
\end{array}
\right.
$$
The work of Wolff, later simplified by Erdo\u{g}an \cite{E1},  improved upon a lower bound due to Bourgain \cite{Bou1} who was the first to bring Fourier restriction theory to bear on the problem. In higher dimensions, the best known lower bounds are
$$
\beta_d(\alpha)\ge\left\lbrace
\begin{array}{lll}
\alpha,& \alpha\in(0,\frac{d-1}{2}],& \\
&& \text{(Mattila~\cite{M0})}\\
\frac{d-1}{2},& \alpha\in[\frac{d-1}{2},\frac{d}{2}],& \\
&&\\
\alpha-1 +\frac{d+2-2\alpha}{4},& \alpha\in[\frac{d}{2},\frac{d+2}{2}],&  \text{(Erdo\u{g}an~\cite{E3,E2})}\\
&&\\
\alpha-1, &\alpha\in[\frac{d+2}{2},d],&\text{(Sj\"olin~\cite{S})}.
\end{array}
\right.
$$

On the other hand, by considering limits of very simple measures supported on small sets; see for example \cite[Chapter 15.2]{M}, 
it is easy to show that 
$$
\beta_d(\alpha)\le\left\lbrace
\begin{array}{ll}
\alpha,& \alpha\in(0, d-2],\\
&\\
\alpha-1+\frac{d-\alpha}{2},& \alpha\in[d-2,d].
\end{array}
\right.
$$
The second bound, for larger $\alpha$, is given by what is  known as  the \lq Knapp example'.
We see that the difference between the best known upper and lower bounds is never more than one and the bounds coincide when $\alpha<\frac{d-1}{2}$ or $\alpha=d$.
Worse counterexamples have been constructed for signed measures by Iosevich and Rudnev~\cite{IR}, or when  
 the averages are taken over a piece of
paraboloid rather than the sphere by Barcel\'o, Bennett, Carbery, Ruiz and Vilela \cite{BBCRV}. Indeed, there is an extensive literature regarding  averages over different manifolds and other generalisations; see for example \cite{BGGIST, BHI, HL,  HI, SS0} and the references therein. 

We will first prove the following upper bound.

\begin{theorem}\label{new} Let $d\ge 4$. Then
$$ \beta_{d}(\alpha) 
\leq \alpha - 1 +\frac{2(d-\alpha)}{d}.
$$
\end{theorem}

 This improves  the previous upper bound  when $\alpha>d/2$ and $d\ge 5$ (it  coincides  precisely with the bound given by the  Knapp example  when $d=4$). For this we will take advantage of a well-known number theoretic result which counts the number of ways the square of a large integer can be represented as a sum of squares.

The bulk of the article will then be dedicated to proving the following lower bound.
\begin{theorem}\label{us}
Let $d \geq 3$. Then
\begin{equation*}\label{MainResultErd}
\beta_{d}(\alpha) 
\geq \alpha - 1
+  \frac{(d-\alpha)^2}{(d-1)(2d-\alpha-1)}.
\end{equation*}
\end{theorem}

This improves the estimate of Sj\"olin for all $\alpha<d$ and the estimate of Erdo\u{g}an for\footnote{in fact in a very slightly larger range. 
%$\alpha>\frac{5d^2-(d-1)\sqrt{9d^2-40d+48}-13d}{4(d-3)}$.
} $\alpha\ge d/2+2/3+1/d$.  
This is not enough to improve the state-of-the-art for the Falconer distance set conjecture (the argument of Mattila \cite{M0} combined with Theorem~\ref{us} implies that distance sets associated to $\alpha$-dimensional sets have positive Lebesgue measure whenever $\alpha> d/2+5/12$). On the other hand, the difference between the best known upper and lower bounds is now strictly less than  $5/6$, from which we can deduce new information regarding the  pointwise convergence of solutions to the wave equation. 

Considering $\partial_{tt} v=\Delta v$ on $\R^{d+1}$, with $v(\cdot,0)=v_0$ and $\partial_tv(\cdot,0)=v_1$, we take the initial data in the homogeneous space $\dot{H}^s\times \dot{H}^{s-1}$, where $$
\dot{H}^{s} := \left\{ 
I_{s} * f \ : \ f \in L^{2}(\R^d)
\right\}.
$$
Here $I_{s}$ is the Riesz kernel defined by $\widehat{I}_{s} := |\cdot|^{- s}$. The almost everywhere convergence question was first considered by Cowling \cite{Cow}, who proved 
\begin{equation*}\label{CowlingProblem}
\lim_{t\to0} v(x , t) = v_{0}(x), \qquad \text{a.e.} \quad x \in \R^d, \qquad \forall \  (v_0,v_1)\in \dot{H}^s\times \dot{H}^{s-1}
\end{equation*}
as long as $s>1/2$. Walther \cite{Wa} then proved that this is not true when $s\le 1/2$, and so the Lebesgue measure question is completely solved for the wave equation. As before we write
\begin{equation*}
\mathcal{D}(v_{0},v_1) := \Big\{\, 
x \in \R^d\ : \lim_{t\to0} v(x, t) \neq v_{0}(x)
\,\Big\},
\end{equation*}
and consider the refined problem of providing upper bounds for 
\begin{equation*}
\gamma_{d}(s) := 
\sup_{(v_{0},v_1)\in \dot{H}^{s}\times \dot{H}^{s-1}} \dim_{H}\big(\mathcal{D}(v_{0},v_1) \big).
\end{equation*}
Sharp estimates were proven in the two-dimensional case in \cite{BBCR}, using the following proposition which forms the link with the decay estimate \eqref{dz}.
\begin{proposition}\label{propo} 
Let $d\ge 2$ and $0<s<d/2$. Then $\beta_d(\alpha)> d-2s\ \Rightarrow \  \gamma_{d}(s)\le \alpha.$
\end{proposition}
 Estimates for the inhomogeneous spaces $H^s(\R^d)$ were proven in \cite{BBCR}, which puts unnecessary restrictions on the data $v_1$, but we will see that the implication also holds in this slightly more general context. Using Sj\"olin's bound  $\beta_d(\alpha)\ge \alpha-1$ they deduced that $\gamma_d(s)\le d+1-2s$, so a consequence of Theorem~\ref{us} is that $\gamma_d(1)< d-1$, ruling out divergence on spheres if the initial data belongs to the energy space $\dot{H}^1(\R^d)\times L^2(\R^d)$.

The exponent $\beta_d(\alpha)$ is also connected to dimension estimates for orthogonal projections; see for example the recent work of Oberlin--Oberlin \cite{OO}. For a related problem regarding Fourier convergence at the points where the function is zero, see \cite{CS} or \cite{CGV} and the references therein.

Although Theorem~\ref{us} yields new bounds for the Schr\"odinger equation, via an appropriate version of Proposition~\ref{propo}, those presented in Theorem~\ref{bil} follow by a more direct use of the techniques developed to prove Theorem~\ref{us}. Compared to the cone,  the paraboloid has an extra nonzero principal curvature, and so it is not always efficient to use Proposition~\ref{propo} in that case. For this reason we have presented the results for the Schr\"odinger equation in $\R^{n+1}$ and the results for the wave equation in~$\R^{d+1}$, where $d=n+1$, and this convention will be maintained throughout. 

The key ingredient in the proofs of Theorems \ref{bil} and \ref{us} will be the multilinear extension estimate due to Bennett, Carbery and Tao \cite{BCT}, which was first successfully employed to prove linear estimates by Bourgain and Guth \cite{BG}. We present the multilinear estimates in the Section~\ref{mult} and a decomposition due to Bourgain and Guth in Section~\ref{decomposition}. In Section~\ref{fourierdecay} we prove Theorem~\ref{us} and in Section~\ref{schrodingerconv} we prove Theorem~\ref{bil}. In Section~\ref{prop} we present the simple proof of Proposition~\ref{propo}, via polar coordinates. In the following section we prove our upper bound for $\beta_d(\alpha)$, using the number theoretic result.

\section{Proof of Theorem~\ref{new}}

This example is inspired by that of \cite{BBCRV}, however they remark that their arguments for the paraboloid do not appear to extend in a routine manner to the sphere.

We will require the following lemma.

\begin{lemma}\label{lemma:Meas1Tris}
Let $0<\alpha \leq d$ and $0 < \varepsilon, \kappa< 1$. For all $R >1$, define 
$$
\Lambda := \Big( R^{\kappa-1} \mathbb{Z}^{d} + B(0,\varepsilon R^{-1} ) \Big) \cap B(0,1)
$$
and
$d \mu := \chi_{\Lambda} dx$, where $dx$ is the Lebesgue measure on $\mathbb{R}^{d}$. 
Then
\begin{equation}\label{Formula:lemma:Meas1Tris}
c_{\alpha}(\mu)
\lesssim \max \big( R^{-d\kappa}, R^{\alpha-d}\big).
\end{equation}
\end{lemma}

\begin{proof}
Notice that $\Lambda$ is the union of approximately $R^{d(1-\kappa)}$ balls of radius $\varepsilon R^{-1}$
whose centres are pairwise separated by $R^{\kappa-1}$.
We consider different cases depending on the size of $r$.

When $0 < r \leq \varepsilon R^{-1}$, the ball $B(x,r)$ overlaps with only one ball $B_j$ of $\Lambda$. Thus
\begin{eqnarray}\nonumber
r^{-\alpha} \mu(B(x,r)) 
& \le & 
r^{-\alpha} |B(x,r) \cap B_j|\\\nonumber
&\lesssim& r^{-\alpha} \min (r^{d}, R^{-d})
\\ \nonumber
& \le &
r^{d-\alpha} \lesssim R^{\alpha-d}.
\end{eqnarray} 
On the other hand, if $\varepsilon R^{-1} < r \leq R^{\kappa-1}$, then $B(x,r)$ overlaps with at most $ N \lesssim 2^{d}$ balls $B_j$, with $j = 1, \dots, N$, contained in $\Lambda$. Thus
\begin{eqnarray*}
r^{-\alpha} \mu(B(x,r)) 
&\leq&  
r^{-\alpha} \Big| B(x,r) \bigcap  \bigcup_{j=1,\dots, N} B_j \Big|\\
&\lesssim& 
r^{-\alpha} | B_N|\lesssim r^{-\alpha} R^{-d}\lesssim R^{\alpha-d}.
\end{eqnarray*} 
Finally, if  $R^{\kappa-1} < r$, then $B(x,r)$ overlaps with at most $N \lesssim \min(r^{d},1)R^{d(1-\kappa)}$ balls $B_j$, with $j = 1, \dots, N$, contained in $\Lambda$. In this case, 
\begin{eqnarray}\nonumber
r^{-\alpha} \mu(B(x,r)) 
& \leq &  
r^{-\alpha} 
\Big| B(x,r) \bigcap  \bigcup_{j=1,\dots, N} B_j \Big|
\\ \nonumber
& \lesssim & r^{-\alpha} N R^{-d}
\lesssim r^{-\alpha} \min(r^{d},1) R^{-d\kappa}.
\end{eqnarray} 
Now as
\begin{equation}\nonumber
r^{-\alpha} \min (r^{d}, 1) R^{-d\kappa} =
\left\{
\begin{array}{lcll}
r^{-\alpha+d} R^{-d\kappa} &\leq R^{-d\kappa} & \mbox{if} & r \leq 1 
\\
r^{-\alpha}R^{-d\kappa}  &\leq R^{-d\kappa} & \mbox{if} & r > 1, 
\end{array}
\right.
\end{equation}
 by collecting the three cases, the proof is complete.
\end{proof}

Let $\sigma$ denote the surface measure on $\mathbb{S}^{d-1}$, and write  $g=g_1-g_2+i(g_3-g_4)$, where each component $g_j$ is positive. Then by considering the positive measures $g_j\mu$, an application of the triangle inequality combined with \eqref{dz} tells us that
\begin{equation*}
\|\widehat{g\mu}(R\,\cdot\,)\|^2_{L^2(\mathbb{S}^{d-1})}\lesssim
c_{\alpha}(\mu)\|\mu\|R^{-\beta}\|g\|_{L^\infty(\mathbb{S}^{d-1})}
\end{equation*}
Thus, by duality, 
we are looking for an upper bound for the $\beta$ such that 
\begin{equation}\label{dual}
\|(f d \sigma)^\vee (R \, \cdot \,) \|_{L^{1}(d \mu)}
\lesssim 
R^{-\beta/2} \sqrt{c_{\alpha}(\mu) \| \mu \|} \| f \|_{L^{2}(\mathbb{S}^{d-1})},
\end{equation}
where
\begin{equation}\nonumber
(f d \sigma)^\vee (x) =\frac{1}{(2\pi)^{d/2}} \int_{\mathbb{S}^{d-1}} e^{i \omega \cdot x} f(\omega) d \sigma (\omega).
\end{equation}
We test this on the characteristic function  associated to $\Omega$ defined by \begin{equation}\nonumber
\Omega = \{ \omega \in \mathbb{S}^{d-1} \ :\ \dist (\omega, \Gamma) \leq \rho R^{-1}  \},
\end{equation}
where $\rho>0$ is sufficiently small, to be chosen later, and $\Gamma$ is defined by
\begin{equation}\nonumber
\Gamma = \{ \omega \in \mathbb{S}^{d-1} \ :\ R^{\kappa} \omega \in 2\pi\mathbb{Z}^{d} \},
\end{equation}
with  $0<\kappa<1$.
Considering  $f$ defined by
\begin{equation}\nonumber
f = \frac{\chi_{\Omega}}{\sqrt{ \sigma (\Omega)  }},
\end{equation}
we have  that $\| f \|_{L^{2}(\mathbb{S}^{d-1})} = 1$.

Now it is well-known (see for example \cite{F}) that for $d\ge 4$, there is a lower bound
\begin{equation}\nonumber
\# \Gamma \gtrsim R^{\kappa (d-2)},
\end{equation}
as long as $R$ is large enough and satisfies $(\tfrac{1}{2\pi}R^{\kappa})^2 \in \mathbb{N}$. Thus, for these values of~$R$, we have
\begin{equation}\label{WHU1Tis}
\sigma (  \Omega )  \gtrsim R^{\kappa (d-2) - (d-1)}.
\end{equation}
We claim that
\begin{equation}\label{rtTris}
|(f d \sigma)^\vee (Rx)|=\left|
\int_{\mathbb{S}^{d-1}}e^{ i \omega \cdot R x} f(\omega) d \sigma (\omega) 
\right|
\gtrsim \sqrt{\sigma ( \Omega )} 
\quad
\quad \forall \
x \in \Lambda,
\end{equation}
where, taking $\varepsilon$ sufficiently small, 
$\Lambda$ is defined by
\begin{equation}\nonumber
\Lambda = \Big(R^{\kappa - 1} \mathbb{Z}^{d} + B(0, \varepsilon R^{-1} )\Big) \cap B(0,1).
\end{equation}
The idea is that the phase of the integrand in \eqref{rtTris} never strays too far from zero modulo~$2\pi i$, and so the different pieces of the integral, corresponding to different pieces of~$\Omega$,  do not cancel each other out. 

More precisely we prove that
\begin{equation}\label{I*Tris}
\omega \cdot R x \in 2\pi \mathbb{Z} +(-\tfrac{1}{10},\tfrac{1}{10}),
\end{equation}
provided that $\omega \in \Omega$ and $x \in \Lambda$.  To see this, we write
\begin{equation}\nonumber
\omega = 2\pi R^{-\kappa} \ell   + v, \qquad \text{where}\quad
\ell \in \mathbb{Z}^{d}, \ \ | \ell | = \tfrac{1}{2\pi}R^{\kappa}, \ \ |v| < \rho R^{-1} 
\end{equation}
and
\begin{equation}\nonumber
x =   R^{\kappa-1}m + u, \qquad \text{where}\quad
m \in \mathbb{Z}^{d}, \ \  |m| < R^{1-\kappa}, \  \ |u| < \varepsilon R^{-1}, 
\end{equation}
so that
\begin{eqnarray*}
\omega \cdot R x 
& = &
( 2\pi R^{-\kappa} \ell + v) \cdot
(  R^{\kappa}m + R u)
\\ \nonumber
& = &
  2\pi \ell  \cdot m
+   v \cdot R^{\kappa} m
+   2\pi R^{1-\kappa} \ell \cdot u
+ v \cdot R u
\\
\nonumber 
& =: & I_{1}+I_{2}+I_{3}+I_{4}.
\end{eqnarray*}
Since $I_{1} \in 2\pi \mathbb{Z}$ and $| I_{2} | \leq  \rho R^{-1} R^{\kappa} R^{1-\kappa} =  \rho$, 
\begin{equation}\nonumber
| I_{3} | <    R^{1-\kappa}R^{\kappa} \varepsilon R^{-1} =   \varepsilon\quad \text{and}\quad
| I_{4} | < \rho R^{-1} R \varepsilon R^{-1} = \rho \varepsilon R^{-1},
\end{equation}
 we see that (\ref{I*Tris}) holds by taking  $\rho$ and $\varepsilon$ sufficiently small. 
This
implies that the phase in \eqref{rtTris} 
 is close enough to zero modulo $2\pi i$ as long as $x \in \Lambda$, yielding the bound.

Now, defining $\mu$ by
$
d \mu = \chi_{\Lambda} dx,$
where $dx$ is the Lebesgue measure in $\mathbb{R}^{d}$, and taking $\kappa = \frac{d-\alpha}{d}$, 
by Lemma~\ref{lemma:Meas1Tris}, we have
\begin{equation}\label{WHU2Tris}
c_{\alpha}(\mu)
\lesssim  R^{\alpha-d} = R^{-d \kappa}.
\end{equation}
On the other hand, 
\begin{equation}\label{WHU3Tris}
\|\mu\| = | \Lambda | \simeq R^{-d}R^{d(1-\kappa)}=R^{-d \kappa}.
\end{equation}
Now \eqref{dual} combined with (\ref{rtTris}) tell us that $$\sqrt{\sigma ( \Omega )}\|\mu\|\lesssim R^{-\beta/2}\sqrt{c_{\alpha}(\mu)\|\mu\|},$$ so that by plugging in  (\ref{WHU1Tis}),  (\ref{WHU2Tris}) and (\ref{WHU3Tris}), we obtain
\begin{equation}\nonumber
R^{\frac{1}{2}(\kappa (d-2) - (d-1))}\lesssim R^{-\beta/2} .
\end{equation}
Letting $R$ tend to infinity, we see that
\begin{equation}
\beta \leq d-1 - \kappa (d-2) 
=
\alpha - 1 + \frac{2(d-\alpha)}{d},
\end{equation}
and so taking $\beta$ sufficiently close to $\beta_d(\alpha)$, the proof is complete.\hfill $\Box$

\section{Multilinear extension estimates}\label{mult}

Here we present the multilinear extension estimates due to Bennett, Carbery and Tao \cite{BCT}. The extension operator, defined below, is also the adjoint of the operator that restricts the Fourier transform to a surface, and so they are also referred to as restriction estimates. 
We consider the surfaces $$S:=\{ (\xi,\phi(\xi))\in \R^{d}\, :\, |\xi|\le 1/2\}$$ with $\phi(\xi)=-|\xi|^2$ or $\phi(\xi)=\sqrt{1-|\xi|^2}$. 
For a cap $\tau= \left\{
(\xi, \phi(\xi) ) : \ \xi \in {Q}
\right\}\subset S$ associated to a cube $Q$,
we define the extension operator $T_{\tau}$ by
\begin{equation*}
T_{\tau} g(x,t) = \int_{{Q}}
g(\xi)\,e^{ix \cdot \xi +it\phi(\xi)}  d\xi,
\end{equation*}
Letting $ Y(\xi)\in \mathbb{S}^{d-1} $ be the outward unit normal vector at a point $(\xi, \phi(\xi))\in S$, 
 we say that the caps $\tau_{{1}}, \dots, \tau_{{m}}$
are $m$-transversal with constant $\theta>0$ if
\begin{equation*}\label{Def:m-Transversality}
|Y(\xi_1) \wedge \dots \wedge Y(\xi_m)|> \theta,
\end{equation*}
for all $\xi_1\in {Q}_1$, \ldots $\xi_m\in {Q}_m$. In the following theorem, and throughout, $B_{R}$ denotes a ball of radius $R$ with arbitrary centre.
\begin{theorem}\label{MultilinEstimates} \cite{BCT}
Let $d\ge 2$, $\varepsilon > 0$ and let $\tau_{{1}} , \dots , \tau_{{d}} \subset S$ be 
$d$-transversal caps with constant $\theta>0$.
Then,  for all $R > 1$,
\begin{equation*}
\bigg\|   
\prod_{k=1}^{d} 
T_{\tau_{{k}}} g
\bigg\|^{\frac{2}{d-1}}_{L^{\frac{2}{d-1}}(B_{R})} \lesssim \mathfrak{c}(\theta)
R^{\varepsilon}
\prod_{k=1}^{d} \| g \|^{\frac{2}{d-1}}_{L^{2}({Q}_{k})}.
\end{equation*}
\end{theorem}

The exact dependence of $\mathfrak{c}$ on $\theta$ is an interesting open question. 
The following version is lower dimensional and it has also been discretised as in  \cite[pp. 1250]{BG}. This is the version we will require in the following section.
\begin{proposition}\cite{BG}\label{MainBouLemma}
Let $0<\varepsilon<\frac{1}{4d}$ and  
let 
$\tau_{{1}}, \dots , \tau_{{m}}\subset \tau$ be
$m$-transversal caps with constant $\theta$, where $2 \leq m \leq d-1$. Let $\mathbb{V}_{\!m}$ be an $m$-dimensional 
subspace of $\mathbb{R}^{d}$  and let ${Q}_{j_k} \subset {Q}_{{k}}$ 
be disjoint cubes of side length $1/K$ such that 
$
\dist \left( Y(\xi), \mathbb{V}_m \right) \leq 1/K
$
for some $\xi\in {Q}_{j_k}$.
Then, for all $K > 1$,
\begin{equation*}\label{LowerDimDiscreteMultilinear} 
\fint_{B_{K}} 
\prod_{k=1}^{m} 
\bigg| \sum_{j_k} 
T_{\tau_{j_k}} g 
\bigg|^{\frac{2}{m-1}} 
 \lesssim \mathfrak{c}(\theta) K^{\varepsilon}
\Bigg(
\fint_{B_{K}}
\prod_{k=1}^{m} 
\bigg(\sum_{j_k}
\big|
T_{\tau_{j_k}} g \big|^{2}
\bigg)^{\frac{1}{2m}}
\Bigg)^{\frac{2m}{m-1}}\!\!\!.
\end{equation*}
\end{proposition}

In fact, due to rescaling arguments we will require these estimates for slightly more general phases $\phi$. Note first that, as we are only interested in the modulus of the extension operator, we are free to add and subtract constants to the phase $\phi$ and so we work instead with $\phi(\xi)=\sqrt{1-|\xi|^2}-1$ in the spherical case so that it looks very similar to the parabolic case. Then, for $\xi_0\in \{ \xi\in \R^{d-1} :  |\xi|\le 1/2-\delta/2\}$ and $0<\delta< 1$, we define the scaling map $S_{\xi_0,\delta}$ by 
$$
S_{\xi_0,\delta}\phi(\xi)=\delta^{-2}\Big(\phi(\xi_0+\delta\xi)-\delta\nabla \phi(\xi_0)\cdot \xi-\phi(\xi_0)\Big).
$$
Note that the paraboloid is unchanged by this operation, and the sphere is changed only very mildly. 
The estimates of this section hold uniformly for all the extension operators defined with a phase obtained by applying the scaling map a finite number of times to $\phi$.

Finally we present a globalised-in-space version of Theorem~\ref{MultilinEstimates} that we will need in the final sections. It follows by a standard localisation argument.

\begin{proposition}\label{MultilinEstimates2}
Let $\varepsilon > 0$, $p=\frac{2d}{d-1}$ and let $\tau_{{1}} , \dots , \tau_{{d}} \subset S$ be 
$d$-transversal caps with constant $\theta>0$. Let $\{\Omega\}$ be a partition of $\R^{d-1}$ into cubes of side length $R$. 
Then,  for all $R > 1$,
\begin{equation*}
\bigg\|   
\prod_{k=1}^{d} 
|T_{\tau_{{k}}} g|^{\frac{1}{d}}
\bigg\|^{2}_{L^{p}(\R^{d-1}\times (-R,R))}\le \sum_{\Omega} \bigg\|   
\prod_{k=1}^{d} 
|T_{\tau_{{k}}} g|^{\frac{1}{d}}
\bigg\|^{2}_{L^{p}(\Omega\times (-R,R))} \lesssim \mathfrak{c}(\theta)
R^{\varepsilon}
\| g \|^2_{2}.
\end{equation*}
\end{proposition}

\begin{proof} Noting that the first inequality is nothing more than the inclusion $\ell^2\subset \ell ^p$, it remains to prove the second which we rewrite as
\begin{equation*}
\sum_{\Omega} \bigg\|   
\prod_{k=1}^{d} 
T_{\tau_{k}} g
\bigg\|^{{2/d}}_{L^{\frac{2}{d-1}}(\Omega\times (-R,R))} \lesssim \mathfrak{c}(\theta)
R^{\varepsilon}
\| g \|^2_{2}.
\end{equation*}
For this we write $g_{\Omega}=\big((g\chi_{Q_k})^\vee\chi_{\Omega^*})^\wedge$ and $g_{\Omega^c}=g\chi_{Q_k}-g_{\Omega}$, where $\chi_{\Omega^*}$ is a  Schwartz function adapted to the cube $\Omega^*$, with same centre as $\Omega$, but with side length \begin{equation}\label{dew}10\sup_{|\xi|\le 1/2} |1+\nabla \phi(\xi)|R.\end{equation} Now that we have taken the support restriction inside the definition of the functions, we will consider the operator $T$ defined by
\begin{equation}\label{fut}
T g(x,t) = \int_{\R^{d-1}} \psi(\xi)
g(\xi)\,e^{ix \cdot \xi +it\phi(\xi)}  d\xi,
\end{equation}
where $\psi$ is a Schwartz function supported in the unit ball and equal to one on $|\xi|\le 1/2$.
By applications of the triangle inequality it would then suffice to  bound the main term as
\begin{equation*}
\sum_{\Omega} \bigg\|   
\prod_{k=1}^{d} 
T g_\Omega
\bigg\|^{{2/d}}_{L^{\frac{2}{d-1}}(\Omega\times (-R,R))} \lesssim \mathfrak{c}(\theta)
R^{\varepsilon}
\| g \|^2_{2},
\end{equation*}
and prove other mixed inequalities, like for example
\begin{equation}\label{mes}
\sum_{\Omega} \bigg\|   
T g_{\Omega^c}\prod_{k=2}^{d} 
T g_\Omega
\bigg\|^{{2/d}}_{L^{\frac{2}{d-1}}(\Omega\times (-R,R))} \lesssim 
\| g \|^2_{2}.
\end{equation}
The main term is bounded directly using Theorem~\ref{MultilinEstimates} and the finite overlapping of the frequency supports. For the second estimate we first note that
by H\"older's inequality, followed by Bernstein's inequality (or Young's inequality given the compact frequency support and the reproducing formula that it yields, see below)  and Plancherel's identity in the $x$-variable, the left-hand side of \eqref{mes} is bounded by
\begin{equation*}
 \sum_{\Omega} \Big(\Big\| \|  
T g_{\Omega^c}\|_{L^2(\Omega)}
\|g_\Omega e^{i t\phi(\cdot)}\|^{d-1}_{2}
\Big\|_{L^{\frac{2}{d-1}}(|t|\le R)}\Big)^{{2/d}}.
\end{equation*}
Then by H\"older's inequality in the time integral, we see that this is bounded by
\begin{equation*}
 R^\frac{d-2}{d}\sum_{\Omega} \Big( \|  
Tg_{\Omega^c}\|_{L^2(\Omega\times (-R,R))}
\|g_\Omega\|^{d-1}_{2}
\Big)^{{2/d}}.
\end{equation*}
A final application of H\"older's inequality in the sum, and the finite overlapping of the frequency supports, shows that this is bounded by 
\begin{equation*}
 R^\frac{d-2}{d}\|g\|_2 \Big(\sum_{\Omega}\|  
T g_{\Omega^c}\|^2_{L^2(\Omega\times (-R,R))}\Big)^{1/2}.
\end{equation*}
Thus in order to complete the proof of \eqref{mes}, we need only prove that
\begin{equation}\label{rt}
\sum_{\Omega}\|  
T g_{\Omega^c}\|^2_{L^2(\Omega\times (-R,R))}\lesssim R^{-N}\|g\|^2_2
\end{equation}
for large enough $N\in\mathbb{N}$.

For this we write the operator as a convolution, 
\begin{equation*}
T g_{\Omega^c}(x,t) = \int_{\R^{d-1}} \psi(\xi)
g_{\Omega^c}(\xi)\,e^{ix \cdot \xi +it\phi(\xi)}  d\xi=\int_{z-x\notin  \Omega^*}\!\!\!\!\!T [1] (z,t) (g\chi_{Q_1})^\vee(z-x) dz.
\end{equation*}
Recalling the definitions \eqref{dew} and \eqref{fut}, we have that $|z| \ge 2|\nabla\phi(\xi)t|$ when $(x,t)\in \Omega\times(-R,R)$, so by repeated integration by parts we see that
$$
|T [1] (z,t)|\lesssim C_N(1+|z|)^{-N-d-1}, 
 $$
so that, for $(x,t)\in \Omega\times(-R,R)$, we have
\begin{equation*}
|T g_{\Omega^c}(x,t)| \lesssim R^{-N-1}\int_{\R^{d-1}} (1+|z|)^{-d} |(g\chi_{Q_1})^\vee(z-x)| dz.
\end{equation*}
Plugging this into \eqref{rt}, and integrating in time, we see that
\begin{align*}
\sum_{\Omega}\|  
T g_{\Omega^c}\|^2_{L^2(\Omega\times (-R,R))}&\lesssim R^{-N}\int_{\R^{d-1}}\Big|\int (1+|z|)^{-d} |(g\chi_{Q_1})^\vee(z-x)| dz\Big|^2 dx\\
&\lesssim R^{-N}\|g\|^2_2,
\end{align*}
where the final inequality is by Young's inequality and the Plancherel identity. This completes the proof of \eqref{rt} and thus \eqref{mes}, and the other mixed terms are bounded in an analogous manner. 
\end{proof}

\section{The Bourgain--Guth decomposition}\label{decomposition}

In order to take advantage of the multilinear estimates, we must first decompose the operator in such a way that transversality presents itself. In order to take advantage of bilinear estimates, this can be done by employing something like a Whitney decomposition.
A triumph of the work of Bourgain and Guth \cite{BG} was to achieve something similar in the multilinear setting. In fact they use the lower dimensional multilinear estimates of Proposition~\ref{MainBouLemma} in order to create the \lq decomposition' (really it is an inequality) and in the coming sections we will need pointwise control this. Indeed we will make essential use of the fact that the right-hand side of the inequality is almost constant at certain scales.
As they point out, this only holds after mollifications, 
and the final decomposition is obtained by an iteration. In this section, we keep track of some of the details that they omitted so as to check that these approximations, as well as the lack of control of the constant~$\mathfrak{c}$ from the previous section, do not feedback in an uncontrolled way.

Let ${Q}\subset \{ \xi\in\R^{d-1} : |\xi|\le 1/2\}$ be a box of side length $\delta$  and let $\tau$ denote the associated cap. 
Take $0 < \varepsilon <\frac{1}{4d}$  
and $R > 1$
and introduce $d$ different scales
\begin{equation*}\label{Scales}
R^{1/\mathfrak{c}(\varepsilon)}< K_{2} < 
\dots<
K_{d+1}<
R^{\varepsilon}
\end{equation*}
that satisfy $K_m^{8m}\mathfrak{c}(K_{m}^{-m}) \leq K^\varepsilon_{m+1}$, where $\mathfrak{c}(\varepsilon)\ge 1/\varepsilon^{2d}$ dominates  the constant from the previous section. As long as it does not blow up at zero in a very unexpectedly fast way, it would suffice to  take $K_m\simeq R^{\varepsilon^{2(d+2-m)}}$. One can calculate that we also have  $R^{\frac{1}{\varepsilon \mathfrak{c}(\varepsilon)}}K_m\le K_{m+1}$.

Take a partition $\{ {Q}_{2,\ell} \}$ 
of ${Q}$
made of pairwise disjoint
cubes
of side length $\delta/K_{2}$ and 
centered in $\xi_{\ell}$.
Then, for all $m = 3, \dots, d$, 
define recursively a
sub-partition
$\{ {Q}_{m,j} \}$
made by pairwise disjoint cubes of side length $\delta/K_{m}$ and
centered at
$\xi_{j}$  
in such a way that
for every ${Q}_{m,j}$ 
there exists an ${Q}_{m-1, \ell}$
that contains~it. For this we need to suppose that $R>2^{\mathfrak{c}(\varepsilon)}$ in order to have room to choose the scales appropriately, and so this is assumed from now on.

We say that the caps $\tau_{m,j}$ associated to ${Q}_{m,j}$ are at scale $\delta/K_{m}$. 
Recalling that
$$
T_{\tau_{m,j}} g(x , t)  = 
\int_{{Q}_{m,j}} \!g(\xi)\,e^{i x \cdot \xi +it \phi(\xi)} d \xi,
$$
for each $m = 2, \dots, d$, we have 
$$
T_{\tau} g   =
\sum_{j}  T_{\tau_{m,j}} g .
$$
We will also need a restricted version of  $T_{\tau_{m,j}}$.
 Let $\mathbb{V}_{\!m}$ be an $m$-dimensional subspace of  $\mathbb{R}^{d}$
and define  
\begin{align}\nonumber
 {T}^{\mathbb{V}_{\!m}}_{\tau_{m,j}} g 
&:=  \sum_{ \substack{ \upsilon \subset \tau_{m,j} \\ \upsilon \in V_{\tau, m} } } {T}_{\upsilon} g,  
\end{align}
where
$V_{\tau, m}
:=
\left\{\, \tau_{m+1,\ell}\subset \tau \ : \ \dist ( Y(\xi) , \mathbb{V}_{\!m} ) 
\leq \delta/K_{m+1} \ \text{for some}\ \xi\in {Q}_{m+1,\ell}\,\right\}.
$

The following pointwise estimate \cite[pp. 1256]{BG} 
will be a key ingredient:
\begin{eqnarray} \label{ZeroBouFormula}
& |T_{\tau} g(x,t)| &
  \lesssim  K_{d}^{2d}\!\!\! \max_{\tau_{1}, \dots , \tau_{d}\subset \tau} 
\prod_{k=1}^{d} | T_{\tau_{k}}g(x,t) |^{\frac{1}{d}}
\\  \nonumber
 &&\qquad\quad +  \sum_{m=2}^{d-1} K_{m}^{2m}\!\!\!\!\! 
\max_{\substack{\mathbb{V}_{\!m} \\ \tau_{1}, \dots , \tau_{m}\subset \tau}} 
\prod_{k=1}^{m} 
| T^{\mathbb{V}_{\!m}}_{\tau_{k}}g(x,t) |^{\frac{1}{m}}
+
\sum_{m=2}^{d}\max_{\tau_{m}\subset \tau} 
| T_{\tau_{m}} g(x,t) |.
\end{eqnarray} 
Here the caps 
$\tau_{1}, \dots , \tau_{m} $ in the first two maxima are 
$m$-transversal
at scale $\delta/K_{m}$, and the final maximum is over caps $\tau_{m}$ at scale $\delta/K_{m}$.
 This is proved by iterating the following dichotomy: either the operator is bounded by a product of $m+1$ operators associated to transversal caps, or it is not, in which case, given $m$ caps where the operator is large and the hyperplane $\mathbb{V}_m$ that their normals generate, the operators associated to the caps with normal lying outside of  $\mathbb{V}_m$ must be small.

The uncertainty principle tells us that the terms 
should be essentially constant at different scales $\delta/K$.
This can be formalised by replacing them
with suitable majorant functions. Indeed, define the dual set $\tau'$ to be the $d$-dimensional cuboid with dimensions $\delta^{-1}\times \ldots\times \delta^{-1}\times \delta^{-2}$ centred at the origin, and with long side normal to $\tau$ (pointing in the direction of  the normal $Y_\tau$ to the centre of  a cap $\tau$).  The scaled version $K\tau'$ denotes the similar set but with dimensions $K\delta^{-1}\times \ldots\times K\delta^{-1}\times K\delta^{-2}$.
Let $\widehat{\psi}=\widehat{\psi}_{\mathrm{o}}\ast \widehat{\psi}_{\mathrm{o}}$ be a smooth radially symmetric 
cut-off function, supported on $B(0,d) \subset \mathbb{R}^{d}$ and equal to one on $B(0,\sqrt{d}) \subset \mathbb{R}^{d}$ and let $\psi_{K\tau'}$ denote the scaled  version of  $\psi$ adapted to $K\tau'$. By this we mean that 
\begin{equation}\label{Def:EtaAlphaStar}
\psi_{K\tau'} (x, t) := 
\frac{\delta^{d+1}}{K^{d}}
\psi 
\left( \frac{\delta x'}{K} , \frac{\delta^2 t'}{K}   \right),\qquad (x',t')=\Lambda_\tau (x,t),
\end{equation}
where $\Lambda_{\tau} \in SO(d)$ and 
$\Lambda_{\tau}(Y_{\tau}) = (0, \dots , 0, 1)$.
By the modulated reproducing formula, 
\begin{equation*}
| T_{\tau}g | \leq 
| T_{\tau} g | * 
|\psi_{\tau'}|.
\end{equation*}
and one can also calculate (see Lemma \ref{TechnicLemma2} of  the appendix) that
\begin{equation*}
| T_{\tau}g | \lesssim 
\left(
| T_{\tau} g |^{\frac{1}{m}} * 
| \psi_{\tau'}|^{\frac{1}{m}} 
\right)^{m},
\end{equation*}
for any $m\ge 1$. This yields
\begin{equation*}
| T_{\tau}g |^{\frac{1}{m}} \lesssim 
| T_{\tau} g |^{\frac{1}{m}} * 
\zeta_{\tau'},\qquad \zeta (x,t) :=
\left(1 +  |x|^{2} 
+
|t|^{2} \right)^{- \mathfrak{c}(\varepsilon)},
\end{equation*}
and as $\zeta_{\tau'}$ is essentially constant on translates of  $\tau'$, which is a property that is preserved under convolution, we have majorised by an essentially constant function. 
By elementary trigonometry one sees that $\frac{1}{2}K_m\tau'\subset \upsilon'$ whenever  $\upsilon\subset\tau$ is at scale $\delta/K_m$,  so that the dual of  the latter is contained in the former and so
$$
| T^{\mathbb{V}_{\!m}}_{\upsilon}g |^{\frac{1}{m}}\lesssim | T^{\mathbb{V}_{\!m}}_{\upsilon}g |^{\frac{1}{m}}*\zeta_{K_m\tau'}.
$$
Using these observations, \eqref{ZeroBouFormula} can be rewritten as
\begin{eqnarray} \label{FirstBouFormula}
& |T_{\tau} g| &
  \lesssim  K_{d}^{2d}\!\! \max_{\tau_{1}, \dots , \tau_{d}} 
\prod_{k=1}^{d} | T_{\tau_{k}}g |^{\frac{1}{d}}*\zeta_{\tau_k'}
\\  \nonumber
 &&+  \sum_{m=2}^{d-1} K_{m}^{2m}\!\!
\max_{\substack{\mathbb{V}_{\!m} \\ \tau_{1}, \dots , \tau_{m}}} 
\prod_{k=1}^{m} 
| T^{\mathbb{V}_{\!m}}_{\tau_{k}}g |^{\frac{1}{m}}*\zeta_{K_{m}\tau'}
+
\sum_{m=1}^{d-1}\max_{\tau_{m+1}} 
| T_{\tau_{m+1}} g |*\zeta_{\tau_{m+1}'},
\end{eqnarray} 
where as before
$\tau_{1}, \dots , \tau_{m}\subset \tau $ are 
$m$-transversal caps
at scale $\delta/K_{m}$ and the maximum in the last term is
taken over caps of  size $\delta/K_{m+1}$.

\begin{remark}\label{RemOnV}
The maximum over $\tau_{1}, \dots , \tau_{m}, \mathbb{V}_{\!m}$ depends on the value of  $(x,t)$, but
 we can now choose the same 
$\tau_{1}, \dots , \tau_{m}, \mathbb{V}_{\!m}$ for all $(x,t)$
in  a translate of  $K_m\tau'$. In fact, given the dichotomy with which the initial decomposition is obtained, $\mathbb{V}_{\!m}$ can be chosen to be the same in any translate of $K_{m+1}\tau'$. This is because we only need to consider this lower dimensional case in the absence of $m+1$ transversal caps for which the operator is large. These caps are at scale $K_{m+1}$ and so the definition of the subspace $\mathbb{V}_{\!m}$ can be taken uniformly at that scale.
\end{remark}

\begin{definition} Set $\Phi_{\tau,\mathbb{V}_1,\tau_{2}}=1$ and, for $m = 2, \dots, d-1$, define
\begin{equation*}\label{PhiDef}
\Phi_{\tau,\mathbb{V}_m,\tau_{m+1}}
:=
\frac{
K_{m}^{2m} \max_{\tau_{1}, \dots , \tau_{m}\subset \tau}
\prod_{k=1}^{m} 
| T^{\mathbb{V}_{\!m}}_{\tau_{k}}g|^{\frac{1}{m}}*\zeta_{K_{m}\tau'}
+
| T_{\tau_{m+1}} g |*\zeta_{\tau_{m+1}'}}
{\left( \sum_{\upsilon \in V_{\tau, m}\cup\{\tau_{m+1}\}} 
\left(
|T_{\upsilon}g|*\zeta_{\upsilon'}
\right)^{2} \right)^{1/2}
+ R^{-1/\varepsilon}\| g \|_{L^{2}}}.
\end{equation*}
\end{definition}
\noindent With this function, the decomposition
(\ref{FirstBouFormula}) can be rewritten as
 \begin{align}\label{SecondBouFormulaConTThirdLine} |T_{\tau} g| &
 \lesssim  K_{d}^{2d}\!\!\max_{\tau_{1}, \dots , \tau_{d}\subset \tau} 
\prod_{k=1}^{d} |T_{\tau_{k}} g|^{\frac{1}{d}} * \zeta_{\tau_k'} 
\\   \nonumber
&\quad +  \sum_{m=1}^{d-1}  \max_{\mathbb{V}_{\!m},\tau_{m+1}}
\Phi_{\tau,\mathbb{V}_m,\tau_{m+1}}
\Bigg(
\sum_{\upsilon \in V_{\tau, m}\cup\{\tau_{m+1}\}} 
\big(
|T_{\upsilon} g| * 
\zeta_{\upsilon'}\big)^{2} \Bigg)^{1/2}   
+ \mathcal{R}_{\tau}(g),
\end{align} 
where the remainder term $\mathcal{R}_{\tau}(g)$ is defined by
\begin{equation*}\label{Def:Reminder}
\mathcal{R}_{\tau}(g) = R^{-1/\varepsilon}\sum_{m=1}^{d-1}\max_{\mathbb{V}_{\!m},\tau_{m+1}}
\Phi_{\tau,\mathbb{V}_m,\tau_{m+1}}
\| g \|_{L^{2}}.
\end{equation*}
Although $\Phi_{\tau,\mathbb{V}_m,\tau_{m+1}}$ looks complicated, we will no longer care about its explicit form, and focus instead on its properties. These properties, one of which we prove now using the multilinear extension estimate,  hold uniformly for all hyperplanes $\mathbb{V}_m$ and caps $\tau_{m+1}$ at scale $\delta/K_{m+1}$.

\begin{lemma} \label{Lemma:MainOnPhi}
Let $0<\varepsilon<\frac{1}{4d}$ and $0<\delta\le 1$. Let
$\tau_{1}, \dots , \tau_{m} $ be
$m$-transversal caps at scale $\delta/K_{m}$ and let $\tau$ be a cap at scale $\delta$ that contains them.
Then, for all $\mathbb{V}_m\subset \R^{d}$ and $a \in \mathbb{R}^{d}$,
\begin{eqnarray*}\label{MeanPhiBoundBisUTILEINLemma}
& & 
 \fint_{a+K_{m+1}\tau'} 
 \left( \prod_{k=1}^{m} 
|T^{\mathbb{V}_{\!m}}_{\tau_{k}}g|^{\frac{1}{m}}*\zeta_{K_m\tau'}
\right)^{\frac{2m}{m-1}}    
\\  
    & \lesssim  &\label{MeanPhiBoundBisUTILEINLemmaSecTerms}
\mathfrak{c}(K_m^{-m})K_{m+1}^{\varepsilon}     
     \Bigg( \sum_{\upsilon \in V_{\tau,m}} \Big(|T_{\upsilon}g|*\zeta_{\upsilon'}\Big)^2(a)
     \Bigg)^{\frac{m}{m-1}}
+ \Big(R^{-1/\varepsilon}\| g \|_{L^{2}}\Big)^{\frac{2m}{m-1}}.\nonumber
\end{eqnarray*}
\end{lemma}

\begin{proof} Denoting $q:=\frac{2m}{m-1}$,
by the 
trivial bound
$\|T_{\tau_{m+1}}g\|_{L^{\infty}} \leq K_{m+1}^{-n/2}\|g\|_{L^{2}}$, and the definition of  the $K_m$, this would follow from the slightly stronger estimate
\begin{align}\label{stonger}
&\quad \fint_{a+K_{m+1}\tau'} 
 \left( \prod_{k=1}^{m} 
|T^{\mathbb{V}_{\!m}}_{\tau_{k}}g|^{\frac{1}{m}}*\zeta_{K_m\tau'}
\right)^{q}    
\lesssim \\  
    &  
\mathfrak{c}(K_m^{-m})K_{m+1}^{\varepsilon}     
     \Bigg( \sum_{\upsilon \in V_{\tau,m}} \!\!\Big(|T_{\upsilon}g|*\zeta_{\upsilon'}\Big)^2(a)
     \Bigg)^{q/2}
\!\!\!\!\!\!\!+ \Big(\Big(\frac{K_m}{K_{m+1}}\Big)^{\mathfrak{c}(\varepsilon)}K_{m+1}^{n/2}\max_{\upsilon\in V_{\tau,m}}\|T_{\upsilon}g\|_{\infty}\Big)^{q}\!\!\!.\nonumber
\end{align}
By scaling as in the proof of the forthcoming Lemma~\ref{Fuman}, it will be enough to prove this with $\delta=1$, so we can replace $a+K_{m+1}\tau'$ by $B_{K_{m+1}}$ centred at $a$. 
By H\"older's inequality
% \ref{TechnicLemma4} 
and Fubini's theorem, we see that
\begin{eqnarray}\nonumber
&  & \fint_{B_{K_{m+1}}}  \left( \prod_{k=1}^{m} \int 
| T^{\mathbb{V}_{\!m}}_{\tau_{k}} g((x,t)-y_{k})|^{\frac{1}{m}} 
\zeta_{K_m\tau'}^{1/q}(y_{k})\zeta_{K_m\tau'}^{1-1/q}(y_{k}) \ dy_{k}   \right)^{q}  \ \!\!\!dxdt  \\ \nonumber
& \lesssim &  \int    \Bigg( \fint_{B_{K_{m+1}}}  
 \prod_{k=1}^{m} 
| T^{\mathbb{V}_{\!m}}_{\tau_{k}} g((x,t)-y_{k})|^{\frac{2}{m-1}} \ dxdt \Bigg)w(y)\,dy,
\end{eqnarray}
where $ \prod_{k=1}^{m}\zeta_{K_m\tau'}(y_{k})  \, dy_1\ldots dy_{m}=:w(y)dy$
Then by Proposition \ref{MainBouLemma} 
(with $K = K_{m+1}\ \text{and}\ \theta = K^{-m}_{m}$), H\"older's inequality and Fubini, this is bounded by a constant multiple of
%Lemma \ref{TechnicLemma4}
\begin{align}\nonumber
\quad & \mathfrak{c}(K_m^{-m})K_{m+1}^\varepsilon  \int    \Bigg( 
\fint_{B_{K_{m+1}}}  
 \prod_{k=1}^{m}
\Bigg( \sum_{\substack{ \upsilon \subset \tau_{k} \\ \upsilon \in V_{\tau,m}  } } 
| T_{\upsilon} g((x,t)-y_{k})|^{2} 
\Bigg)^{\frac{1}{2m}} \!\!\!\! dxdt \Bigg)^{q}  
 w(y)\,dy\\ \nonumber
& \lesssim \ \ \mathfrak{c}(K_m^{-m})K_{m+1}^\varepsilon  \fint_{B_{K_{m+1}}}  
\int \prod_{k=1}^{m}
\Bigg( \sum_{ \substack{ \upsilon \subset \tau_{k} \\ \upsilon \in V_{\tau,m}  } } 
| T_{\upsilon} g((x,t)-y_{k})|^{2} \Bigg)^{\frac{q}{2m}}
  w(y)\,dy  dxdt.
\end{align}
By H\"older's inequality again and the reproducing formula, we can bound this as
%and Lemma \ref{TechnicLemma4}
%and H\"older inequality
\begin{eqnarray} \nonumber
 & \lesssim &  \mathfrak{c}(K_m^{-m})K_{m+1}^\varepsilon
 \fint_{B_{K_{m+1}}}
 \prod_{k=1}^{m}
\Bigg(
  \sum_{  \substack{ \upsilon \subset \tau_{k} \\ \upsilon \in V_{\tau,m}  }  } 
 \Big(| T_{\upsilon} g |*\zeta_{\upsilon'}\Big)^{2}  
 * 
\zeta_{K_m\tau'}(x,t) 
\Bigg)^{\frac{q}{2m}}
 \ dxdt 
 \\ \nonumber
 & \leq &  \mathfrak{c}(K_m^{-m})K_{m+1}^\varepsilon \fint_{B_{K_{m+1}}}  
\Bigg( 
  \sum_{ \upsilon \in V_{\tau,m}} 
 \Big(| T_{\upsilon} g |*\zeta_{\upsilon'}\Big)^{2}  
 * 
\zeta_{K_m\tau'}(x,t) 
\Bigg)^{q/2}
 \ dxdt. 
 \end{eqnarray}
Finally we can apply Lemma~\ref{TechnicLemma1} of  the appendix, with $K=K_{m+1}$ and $K'=K_m$, to conclude that this is bounded by a constant multiple of
$$
\mathfrak{c}(K_m^{-m})K_{m+1}^\varepsilon     
     \Bigg( \sum_{\upsilon \in V_{\tau,m}} \Big(|T_{\upsilon}g|*\zeta_{\upsilon'}\Big)^2(a) \Bigg)^{q/2}
\!\!\!\!\!\!\!+ \Big(\Big(\frac{K_m}{K_{m+1}}\Big)^{\mathfrak{c}(\varepsilon)}K_{m+1}^{n/2}\max_{\upsilon} \| T_{\upsilon}g \|_{L^{\infty}}\Big)^{q}\!\!\!.
$$
The chain of  inequalities 
yields \eqref{stonger} and hence the result.
\end{proof}

\begin{properties}\label{MeanPhiBoundBis}
It is clear that $\Phi_{\tau,\mathbb{V}_m,\tau_{m+1}}
$ is essentially constant on translates of  $K_m\tau'$.
Given that by definition
$K_m^{\frac{4m^2}{m-1}}\mathfrak{c}(K_{m}^{-m}) \leq K^\varepsilon_{m+1}$,
Lemma \ref{Lemma:MainOnPhi} yields
\begin{equation*}
 \fint_{a+K_{m+1}\tau'} \Phi^{\frac{2m}{m-1}}_{\tau,\mathbb{V}_m,\tau_{m+1}}
 \lesssim K_{m+1}^{2\varepsilon},
\end{equation*}
where $m = 2, \dots  , d-1$.
By H\"older's inequality
this also implies that
\begin{equation*}\label{MeanPhiBoundFirstBis}
\fint_{a+K_{m+1}\tau'} \Phi_{\tau,\mathbb{V}_m,\tau_{m+1}}
^{\frac{2(d-1)}{d-2}}
\lesssim K_{m+1}^{2\varepsilon},
\end{equation*}
uniformly over all $a\in \R^{d}$, $\mathbb{V}_m\subset \mathbb{R}^{d}$ and $\tau_{m+1}\subset \tau$ at scale $\delta/K_{m+1}$.
\end{properties}

We could have convolved both sides of  (\ref{FirstBouFormula}) with $\zeta_{\tau'}$, before introducing the function $\Phi_{\tau,\mathbb{V}_m,\tau_{m+1}}$. In order to then replace the double convolutions on the right-hand side by single convolutions we again use Lemma~\ref{TechnicLemma1} of  the appendix. Introducing $\Phi_{\tau,\mathbb{V}_m,\tau_{m+1}}$ after this process, we can also write
 \begin{align}\label{ThirdBouFormulaConTThirdLine} |T_{\tau} g|\ast \zeta_{\tau'} &
 \lesssim  K_{d}^{2d}\!\! \max_{\tau_{1}, \dots , \tau_{d}} 
\prod_{k=1}^{d} |T_{\tau_{k}} g|^{\frac{1}{d}} * \zeta_{\tau_k'} 
\\   \nonumber
&\quad +  \sum_{m=2}^{d-1}  \max_{\mathbb{V}_{\!m},\tau_{m+1}}
\Phi_{\tau,\mathbb{V}_m,\tau_{m+1}}
\Bigg(
\sum_{\upsilon \in V_{\tau, m}\cup\{\tau_{m+1}\}} 
\Big(
|T_{\upsilon} g| * 
\zeta_{\upsilon'}\Big)^{2} \Bigg)^{1/2}   
+ \mathcal{R}_\tau.
\end{align} 
As the terms on the right-hand side have the same form as the left-hand side at a different scale, we can iterate this inequality to obtain the following theorem. From now on we write $\tau\sim \delta/K$ if $\tau$ is a cap at scale $\delta/K$.

\begin{definition}\label{Def:PsiRec}
Define $\Psi_{\upsilon}$ recursively by
\begin{equation*}\label{EqDef:PsiRic}
\begin{array}{llll}
\Psi_{\upsilon} := 1  &   & \upsilon \sim 1,  
\\
\Psi_{\upsilon} := \Psi_{\tau}\max_{\mathbb{V}_{\!m},\tau_{m+1}}
\Phi_{\tau,\mathbb{V}_m,\tau_{m+1}}&   &
\upsilon\subset \tau,\ \ \upsilon\sim \delta /K_{m+1},\ \ \tau \sim \delta. 
\end{array}
\end{equation*}
\end{definition}

We keep track of the maximal number of caps in the following sets~$E_\delta$ as this information is used when proving linear restriction estimates. However the cardinality will have no consequence in this article - it will only be important that the caps of these sets are disjoint.

\begin{proposition}\label{THM:SelfIteratedFormua}
Let $0 < \varepsilon <\frac{1}{4d}$ and let $S=\{ (\xi,\phi(\xi))\, :\, |\xi|\le 1/2\}$. Then, for all $N \in \mathbb{N}$, 
\begin{eqnarray}\nonumber
 |T_{\!S} g|   
& \lesssim_N &  
K_{d}^{2d}\!\!\!\!\sum_{ 
K_{2}^{-N} < \delta \leq 1 }
\max_{E_{\delta}}
\Bigg( 
\sum_{ \tau \in E_{\delta}} 
  \Psi_{\tau}^{2}
\left( 
\max_{\tau_{1}, \dots ,  \tau_{d}  \subset \tau }\prod_{k=1}^{d}
|T_{\tau_{k}} g|^{\frac{1}{d}}
*\zeta_{\tau_k'} \right)^{2}   
\Bigg)^{1/2}  
\\   \nonumber
& &+
\sum_{K_{2}^{-N}K_{d}^{-1} < \delta \leq K_{2}^{-N} } 
\max_{E_{\delta}}
\left( \sum_{\tau \in E_{\delta}} 
\Psi_{\tau}^{2}
\left( |T_{\tau} g|*\zeta_{\tau'}  \right)^{2}  
\right)^{1/2}
\\ \label{Contribute}
&  &+
\sum_{K_{2}^{-N}K_{d}^{-1} < \delta \leq 1}  
\max_{E_{\delta}}
\left( \sum_{\tau \in E_{\delta}}\Psi^2_{\tau}\right)^{1/2}\!\!\!\!R^{-1/\varepsilon}\|g\|_{L^2},
\end{eqnarray}
provided $\supp g\subset \{ \xi \inÊ\R^{d-1}\, :\, |\xi|\le 1/2\}$. Here $\delta$ is restricted to taking values of  the form 
 $K_{2}^{-\gamma_{2}} \cdot ... \cdot K_{d}^{-\gamma_{d}}$ with $\gamma_{2}, \dots , \gamma_{d} \in \mathbb{N}\cup\{0\}$ and  $\tau_{1}, \dots , \tau_{d}$ 
are $d$-transversal caps at scale 
 $\delta / K_{d}$. The
sets
$E_{\delta}$ consist of  at most $4^{N}\delta^{2 - d}$ 
disjoint caps at scale $\delta$.\end{proposition}

\begin{proof}
When $N=1$, there is only one term in the sum over $K_{2}^{-N} < \delta \leq 1$ and the inequality
follows from (\ref{SecondBouFormulaConTThirdLine}) at scale one.  So we proceed by induction on $N$. 

Suppose the inequality is true for $N$. Note that if it were not for the upper bound on $\delta$ in the second sum on the right-hand side, the inequality with $N+1$ would immediately follow from the $N$th version. Thus it remains to bound the part of  the sum that appears in the $N$th version that does not appear in the version with $N+1$;
$$
\sum_{K_{2}^{-(N+1)} < \delta \leq K_{2}^{-N} } 
\max_{E_{\delta}}
\left( \sum_{\tau \in E_{\delta}} 
\Psi_{\tau}^{2}
\left( |T_{\tau} g|*\zeta_{\tau'}  \right)^{2}  
\right)^{1/2}.
$$
Applying \eqref{ThirdBouFormulaConTThirdLine} to the summands,  this is bounded by a constant multiple of
\begin{eqnarray}\nonumber
\label{BeingUnrestr1}
& & \!\!\!\!\!\!\!\!\!\!
\sum_{K_{2}^{-(N+1)} < \delta \leq K_{2}^{-N} }
\max_{E_{\delta}}
\Bigg(  
\sum_{\tau \in E_{\delta}} 
\Psi_{\tau}^{2}
\left(
K_{d}^{2d}\max_{\substack{\tau_{1}, \dots , \tau_{d} \subset \tau}} 
 \prod_{k=1}^{d} 
|T_{\tau_{k}}g|^{\frac{1}{d}} 
* \zeta_{\tau_{k}'}
\right)^{2}
\Bigg)^{1/2}\\\nonumber
& + &\!\!\!\!\! \!\!\!\!\!
\sum_{K_{2}^{-(N+1)} < \delta \leq K_{2}^{-N}} 
\max_{E_{\delta}}
\Bigg(
\sum_{\tau \in E_{\delta}}  
\Psi_{\tau}^{2} 
\sum_{m=1}^{d-1}\max_{\mathbb{V}_{\!m},\tau_{m+1}}
\Phi^2_{\tau,\mathbb{V}_m,\tau_{m+1}}
\!\!\!\!\!\!\!\!\!\!\!\sum_{\upsilon \in V_{\tau,m}\cup\{\tau_{m+1}\}}\!\!\!\!\!\!\!\!\! 
\left(
 |T_{\upsilon}g|*\zeta_{\upsilon'} 
\right)^{2}
\Bigg)^{1/2}
\\ \nonumber
& + &\!\!\!\!\!\!\!\!\!\!
\sum_{K_{2}^{-(N+1)} < \delta \leq K_{2}^{-N} } 
\max_{E_{\delta}}
\left( \sum_{\tau \in E_{\delta}} 
\Psi_{\tau}^{2}
\mathcal{R}_\tau^{2}(g)  
\right)^{1/2}.
\end{eqnarray}
Here
$\tau_{1}, \dots  , \tau_{d}$ are $d$-transversal caps of  size $\delta / K_{d}$ and
$V_{\tau,m}$ is the set of  all the caps $\upsilon\subset \tau$ of  size $\delta / K_{m+1}$ 
and such that $\dist (Y(\xi), \mathbb{V}_{\!m}) \leq \delta / K_{m+1}$ for some $\xi$ in the orthogonal projection of  $\upsilon$.
 The first term is clearly acceptable and, by the definitions of  $\Psi_{\tau}$ and $\mathcal{R}_\tau$, we can bound the other two as
\begin{eqnarray}\nonumber
&  &\!\!\!\!\! \!\!\!\!\!
\sum_{K_{2}^{-(N+1)} < \delta \leq K_{2}^{-N}} 
\max_{E_{\delta}}
\Bigg(
\sum_{\tau \in E_{\delta}} \sum_{m=2}^{d-1} 
\sum_{\upsilon \in V_{\tau,m}\cup\{\tau_{m+1}\}}  \Psi_{\upsilon}^{2}
\left(
 |T_{\upsilon}g|*\zeta_{\upsilon'} 
\right)^{2}
\Bigg)^{1/2}
\\ \nonumber
& + &\!\!\!\!\!\!\!\!\!\!
\sum_{K_{2}^{-(N+1)} < \delta \leq K_{2}^{-N} } 
n\max_{E_{\delta}}
\left( \sum_{\tau \in E_{\delta}:\upsilon\subset \tau} 
\Psi_\upsilon^{2} R^{-2/\varepsilon}
\| g \|^2_{L^{2}} 
\right)^{1/2}.
\end{eqnarray}
Using the induction hypothesis again, there are at most $$4^N\delta^{2-d}4(\delta/(\delta/K_m))^{d-2}=4^{N+1}(\delta/K_m)^{2-d}$$ terms in the product $E_{\delta}\times V_{\tau,m}\cup\{\tau_{m+1}\}$, so we shift the scale and bound this  by a constant multiple of  
\begin{eqnarray}\nonumber
&  &\!\!\!\!\! \!\!\!\!\!
\sum_{K_{2}^{-(N+1)}K_{d}^{-1} <\delta \leq K_{2}^{-(N+1)}} 
\max_{E_{\delta}}
\Bigg(
\sum_{\upsilon \in E_{\delta}}  
\Psi_{\upsilon}^{2}
\left(
 |T_{\upsilon}g|*\zeta_{\upsilon'} 
\right)^{2}
\Bigg)^{1/2}
\\ \nonumber
& + &\!\!\!\!\!\!\!\!\!\!
\sum_{K_{2}^{-(N+1)}K_{d}^{-1} <\delta \leq K_{2}^{-(N+1)}} 
\max_{E_{\delta}}
\left( \sum_{\upsilon \in E_{\delta}} 
\Psi_\upsilon^{2}
\right)^{1/2}R^{-1/\varepsilon}\|g\|_{L^2}.
\end{eqnarray}
This is also acceptable and so the proof is complete.
\end{proof}

\begin{definition}\label{length}
If $\tau$ is a cap at scale $ K_{2}^{-\gamma_{2}} \cdot\cdot\cdot K_{d}^{-\gamma_{d}}$
we write $l(\tau):= \sum_{j=2}^{d} \gamma_{j}$.
\end{definition}

The functions $\Psi_\tau$ also have good essentially constant properties, that we record in the following proposition.

\begin{proposition}\label{ThePsi}
Let $0<\varepsilon< \frac{1}{4d}$. Then the functions $\Psi_{\tau}$ are essentially constant at scale one. 
%More precisely
%\begin{equation}\label{AlmostConstancyOfThePsi}
%\Psi_{\tau}(x_{1},t_{1}) 
%\leq  C_{\varepsilon}^{l(\tau)} 
%\Psi_{\tau}(x_{2},t_{2}) 
%\end{equation}
%provided that $|(x_{1},t_{1}) - (x_{2},t_{2})| \le 1$. 
Moreover, for all $a \in \mathbb{R}^{d}$,
\begin{equation*}\label{PsiProperty}
\fint_{a + \tau '} 
\Psi_{\tau}^{\frac{2(d-1)}{d-2}} (x,t) \, dxdt 
\lesssim_{l(\tau)}
%\left( 2^{d}D(\varepsilon)^{2q}E(\varepsilon)  \right)^{l(\tau)} 
| \tau' |^{\varepsilon}.
\end{equation*}
\end{proposition}

\begin{proof}
The essentially constant property is an immediate consequence of the definition and the corresponding property for $\Phi_{\upsilon}$ with $\upsilon\subset \tau$,  so it remains to prove the averaged property.

If $\tau \sim 1$, then $\Psi_{\tau} = 1$ and the estimate is trivially
satisfied. 
If $\upsilon \sim 1/K_{m+1}$, then $\Psi_{\upsilon} = \max_{\mathbb{V}_{\!m},\tau_{m+1}}
\Phi_{\tau,\mathbb{V}_m,\tau_{m+1}}$ where $\upsilon\subset \tau\sim 1$ and we can cover
$a + \upsilon'$ with a family of translates of $K_{m+1}\tau'$ which are essentially balls $B_{j}$
of diameter $K_{m+1}$. We can of course do this in such a way that 
\begin{equation*}\label{AngleConseq} 
\bigcup_{j} B_{j}\subset a +4\upsilon'.
\end{equation*} 
Then we have
\begin{eqnarray*}
\int_{a + \upsilon '} \Psi^{q}_{\upsilon} & \leq &
\sum_{j} \int_{B_{j}} 
\max_{\mathbb{V}_{\!m},\tau_{m+1}}
\Phi^q_{\tau,\mathbb{V}_m,\tau_{m+1}}  
\\ \nonumber 
& \lesssim  & 
%D(\varepsilon)^{q} 
\sum_{j} |B_{j}|\fint_{B_{j}}\max_{\mathbb{V}_{\!m},\tau_{m+1}}
\Phi^q_{\tau,\mathbb{V}_m,\tau_{m+1}}  \end{eqnarray*}
Recalling Remark~\ref{RemOnV}, in fact we have the same $\mathbb{V}_{\!m}$ for all $(x,t)\in B_{j}$. Similarly as $\Phi_{\tau,\mathbb{V}_m,\tau_{m+1}}  $ is essentially constant on $B_{j}$
we can suppose that the maximum is attained on the same $\tau_{m+1}$ for a given $B_{j}$. Thus, taking $q=\frac{2(d-1)}{d-2}$, 
by Property \ref{MeanPhiBoundBis} we obtain
\begin{equation*}\label{Psi:R0LargeEnough}
\int_{a + \upsilon '} \Psi^{q}_{\upsilon}\lesssim  \sum_{B_{j}} |B_{j}|K_{m+1}^{2\varepsilon}
\lesssim  |\upsilon' |K_{m+1}^{2\varepsilon}\le|\upsilon' |^{1+\varepsilon}
\end{equation*}
as claimed.

 We have proved the proposition for $\tau$ such that $l(\tau)=0$ or $1$. Thus we can proceed by induction on this quantity. Supposing that we have the estimate for $\tau$ such that $l(\tau)=N$, it will suffice to prove the estimate for $\upsilon$ such that $l(\upsilon)=N+1$.
That is we suppose that
\begin{equation}\label{InductHP}
\fint_{a + \tau'} \Psi_{\tau}^{q}  
\lesssim
%\left( 2^{d}D(\varepsilon)^{2q}E(\varepsilon)  \right)^{l(\tau)} 
| \tau' |^{\varepsilon}, 
\qquad a \in \mathbb{R}^{d},
\end{equation}
and attempt to prove the same for
$\upsilon$ at scale $\delta/K_{m+1}$ such that $\upsilon \subset \tau $ at scale $\delta$. We cover $a + \upsilon'$
with a family $\left\{ T_{\ell}\right \}$
of pairwise disjoint translates of~$\tau'$ with centres at $(x_\ell,t_\ell)$. We can do this in such a way that \begin{equation}\nonumber
\bigcup_{\ell} T_{\ell} \subset a + 2\upsilon'.
\end{equation} 
%this is a simple consequence of the fact that the 
%angle beetween $\upsilon$ and $\tau$ is $O(K_{m} / K^{ \gamma})$.
As $\Phi^q_{\tau,\mathbb{V}_m,\tau_{m+1}}  $ is essentially constant on $T_{\ell}$, we have
\begin{eqnarray*}\label{PsiPropProof1}
\int_{a + \upsilon '} \Psi^{q}_{\upsilon} & \leq &
\sum_{\ell} \int_{T_{\ell}} \Psi^{q}_{\tau} 
\max_{\mathbb{V}_{\!m},\tau_{m+1}}
\Phi^q_{\tau,\mathbb{V}_m,\tau_{m+1}}    
\\ \nonumber 
& \lesssim  & 
%D(\varepsilon)^{q} 
\sum_{\ell} \max_{\mathbb{V}_{\!m},\tau_{m+1}}
\Phi^q_{\tau,\mathbb{V}_m,\tau_{m+1}}  (x_{\ell}, t_{\ell}) 
|T_{\ell}|
\fint_{T_{\ell}} \Psi^{q}_{\tau}.
\end{eqnarray*}
Then, by  the induction hypothesis (\ref{InductHP}), we see that 
\begin{eqnarray*}
\int_{a + \upsilon '} \Psi^{q}_{\upsilon} & \leq &% \left( 2^{d}D(\varepsilon)^{2q}E(\varepsilon)  \right)^{l(\tau)} 
|\tau '|^{\varepsilon} \sum_{\ell} 
\max_{\mathbb{V}_{\!m},\tau_{m+1}}
\Phi^q_{\tau,\mathbb{V}_m,\tau_{m+1}}  (x_{\ell}, t_{\ell}) 
|T_{\ell}|  
 \\ \nonumber
& \lesssim & 
%\left( 2^{d}D(\varepsilon)^{2q}E(\varepsilon)  \right)^{l(\tau)} D(\varepsilon)^{q}
|\tau '|^{\varepsilon}
\int_{\cup_{\ell} T_{\ell}} \max_{\mathbb{V}_{\!m},\tau_{m+1}}
\Phi^q_{\tau,\mathbb{V}_m,\tau_{m+1}}.
 \end{eqnarray*}
We are now in a similar position as in the case $l(\tau)=1$. We cover $\bigcup_{\ell} T_{\ell}$
with a family $\left\{ T_{j}\right \}$
of disjoint translates of $K_{m+1} \tau'$. 
As the angle between $Y_{\upsilon}$ and $Y_{\tau}$ is bounded by $\delta$, elementary trigonometry tells us that we can do this so that  
\begin{equation*}
\bigcup_{j} T_{j}\subset a +4\upsilon'.
\end{equation*} 
Thus, by Remark~\ref{RemOnV} and Property \ref{MeanPhiBoundBis},
\begin{eqnarray*}\label{PsiPropProof3}
\int_{a + \upsilon '} \Psi^{q}_{\upsilon} & \leq & 
|\tau ' |^{\varepsilon} 
\sum_{j} |T_{j}| \max_{\mathbb{V}_{\!m},\tau_{m+1}}\fint_{T_{j}} 
\Phi^q_{\tau,\mathbb{V}_m,\tau_{m+1}} 
\\ \nonumber
& \lesssim &
%E(\varepsilon)
|\tau ' |^{\varepsilon} K_{m+1}^{2\varepsilon} 
\sum_{j} |T_{j}| 
 \lesssim % 2^{d} E(\varepsilon) 
 |\tau '|^{\varepsilon} K_{m+1}^{2\varepsilon}
 |\upsilon ' |\le |\upsilon ' |^{1+\varepsilon}
\end{eqnarray*}
where in the final inequality we used that  $|\upsilon '|^{1+ \varepsilon}=|\tau '|^{\varepsilon} K_{m+1}^{(d+1)\varepsilon}
 |\upsilon ' |$, and so the proof is complete.
\end{proof}

Returning to the decomposition \eqref{Contribute}, we
stop the iteration at the biggest value of $N$ such that 
$ K_{2}^{N} K_{d}<R^\lambda$, where $\lambda> 0$, so that
\begin{eqnarray}\nonumber
 |T_{\!S} g|   
& \lesssim &  
R^{\varepsilon}\!\!\!\!\sum_{ 
R^{-\lambda} < \delta \leq 1 }
\max_{E_{\delta}}
\Bigg( 
\sum_{ \tau \in E_{\delta}} 
  \Psi_{\tau}^{2}
\left( 
\max_{\tau_{1}, \dots ,  \tau_{d}  \subset \tau }\prod_{k=1}^{d}
|T_{\tau_{k}} g|^{\frac{1}{d}}
*\zeta_{\tau_k'} \right)^{2}   
\Bigg)^{1/2}  
\\   \nonumber
& &+
\sum_{R^{-\lambda}  < \delta \leq R^{-\lambda+\varepsilon} } 
\max_{E_{\delta}}
\left( \sum_{\tau \in E_{\delta}} 
\Psi_{\tau}^{2}
\left( |T_{\tau} g|*\zeta_{\tau'}  \right)^{2}  
\right)^{1/2}
\\ \label{FinalSelfIteratedFormula}
&  &+
\sum_{R^{-\lambda} < \delta \leq 1}  \max_{E_{\delta}}
\left( \sum_{\tau \in E_{\delta}}\Psi^2_{\tau}\right)^{1/2}R^{-1/\varepsilon}\|g\|_{L^2}.
\end{eqnarray}
This is what we call the Bourgain--Guth decomposition \cite[pp. 1259]{BG}. Note that as
$
|\tau' | <R^{(d+1)\lambda},
$
 we have 
\begin{equation}\label{PsiPropertyFINAL}
\fint_{a + \tau '} 
\Psi_{\tau}^{\frac{2(d-1)}{d-2}} 
\lesssim
R^{(d+1)\lambda\varepsilon},
\quad
 a \in \mathbb{R}^{d},
\end{equation}
 
Later we will dispose of the sets $E_\delta$ and take the inner sums in $\tau$ over the full partition of $S$. The outer sum (over the scales at which the partition is taken) has less than $\lambda\mathfrak{c}(\varepsilon)$ terms in it,  where $\mathfrak{c}$ is the constant from the Bennett--Carbery--Tao extension estimate. The inequality recalls the way in which the Whitney decomposition can be used to take advantage of bilinear estimates, stopping at a scale for which easy estimates are available.  The big difference between this and the Whitney decomposition are the functions $\Psi_\tau$, which have reasonably nice properties, but will prove to be  something of a hindrance. Indeed, the easy estimates for the linear terms are no longer so good that we can ignore them completely. Our final bounds are obtained by compromising between the scale $\lambda$  that is good for the multilinear term and that which is good for the linear term.

\section{Proof of Theorem~\ref{us}}\label{fourierdecay}
%Let $B_{r} := \{x \in \mathbb{R}^{d} \ : \ |x| \leq r \} $.
%Let $\mu$ be an $\alpha$-dimensional measure 
%supported in $B_{1} \subset \mathbb{R}^{d}$.
Recall that by duality, the desired estimate \eqref{dz} is equivalent to 
\begin{equation*}
\|(fd \sigma)^\vee (R \, \cdot \,) \|_{L^{1}(d \mu)}
\lesssim 
R^{-\beta/2} \sqrt{c_{\alpha}(\mu) \| \mu \|} \| f\|_{L^{2}(\mathbb{S}^{d-1})}.
\end{equation*}
Thus, by H\"older's inequality, it will suffice to prove
\begin{equation}\label{IV}
\|(f d \sigma)^\vee (R \, \cdot \,) \|_{L^{2}(d \mu)}
\lesssim 
R^{- \beta/2} \sqrt{c_{\alpha}(\mu)} \| f \|_{L^{2}(\mathbb{S}^{d-1})}
\end{equation}
with
\begin{equation*}
\beta
>\alpha - 1
+  
 \frac{(d-\alpha)^2}{(d-1)(2d-\alpha-1)}.
\end{equation*}
Defining the measure $\mu_{R}$ by 
$$
\int \psi (x)\, d\mu_R(x)= \int \psi(x)\, R^\alpha d\mu(x/R)=R^\alpha\int \psi(Rx)\, d\mu(x),
$$
 it is clear that 
$c_{\alpha}(\mu_{R}) = c_{\alpha}(\mu)$,
%Notice also that $\mu_{R}$ is supported in $B_{R} \subset \mathbb{R}^{d}$.
so that \eqref{IV} is equivalent to
\begin{equation}\label{V}
\|(f d \sigma)^\vee\|_{L^{2}(d \mu_{R})}
\lesssim 
R^{\frac{\alpha-\beta}{2} } \sqrt{c_{\alpha}(\mu)} \| f\|_{L^{2}(\mathbb{S}^{d-1})}.
\end{equation}

By a finite splitting, the triangle inequality and  the rotational invariance of the inequality (which holds uniformly for all $\alpha$-dimensional measures $\mu_R$) we can suppose that $\sigma$ is supported on 
$S=\{ (\xi,\phi(\xi))\, :\, |\xi|\le 1/2\},$
where $\phi(\xi)=\sqrt{1-|\xi|^2}$. Defining 
$$g(\xi):=\frac{1}{(2\pi)^{d/2}}\frac{f(\xi,\phi(\xi))}{\sqrt{1-|\xi|^2}},$$ we can write
$$
(f d \sigma)^\vee(x,t)=\int_{|\xi|\le 1/2}g(\xi)\, e^{ix\cdot\xi+it\phi(\xi)}d\xi,
$$
so we see that \eqref{V} is equivalent to
\begin{equation}\label{VI}
\| T_{\!S} g \|_{L^{2}(d \mu_{R})}
\lesssim 
R^{\frac{\alpha-\beta}{2}} \sqrt{c_{\alpha}(\mu)} \| g\|_{L^{2}(\mathbb{R}^{d-1})}.
\end{equation}
For this we will use the Bourgain--Guth decomposition with $\lambda=\frac{d-\alpha}{2d-\alpha-1}$;
\begin{eqnarray}\nonumber
 | T_{\!S}g |    
& \lesssim &  
R^{\varepsilon}\!\!\!\!\sum_{ 
R^{-\lambda} \le \delta \leq 1 }
\Bigg( 
\sum_{ \tau \sim \delta} 
\Big( \Psi_{\tau}\!\!\!\!
\max_{\tau_{1}, \dots ,  \tau_{d}  \subset \tau }\prod_{k=1}^{d}
|T_{\tau_{k}}g|^{\frac{1}{d}}
*\zeta_{\tau_k'} \Big)^{2}   
\Bigg)^{1/2}  
\\   \label{schroform0}
& &+
\sum_{R^{-\lambda}  \le \delta \leq R^{-\lambda+\varepsilon} } 
\left( \sum_{\tau \sim \delta} 
\left(\Psi_{\tau} |T_{\tau}g|*\zeta_{\tau'}  \right)^{2}  
\right)^{1/2}
\\ \nonumber
&  &+
\sum_{R^{-\lambda} \le \delta \leq 1}  
\left( \sum_{\tau \sim \delta}\Psi^2_{\tau}\right)^{1/2}R^{-1/\varepsilon}\|g\|_{L^2(\mathbb{R}^{d-1})},
\end{eqnarray}
which follows from estimate \eqref{FinalSelfIteratedFormula} by summing in $\tau$ over the full partition of $S$ at scale $\delta$ instead of over the restricted subsets $E_\delta$. 

Recalling that there are less that $\lambda \mathfrak{c}(\varepsilon)<\mathfrak{c}(\varepsilon)$ terms in each of the $\delta$-sums, by the triangle inequality, we need only prove estimates which are uniform in $\delta$. Writing $g_\tau:=g\chi_\tau$, if we could prove
\begin{equation}\label{first0}
\Big\| \Psi_{\tau}| T_{\tau}g |*\zeta_{\tau'}  \Big\|_{L^{2}(d \mu_{R})} \lesssim \sqrt{c_{\alpha}(\mu)}\,
R^{\frac{\alpha}{2} - \frac{\alpha - 1}{2}
- 
 \frac{\lambda(d-\alpha)}{2(d-1)}+d\varepsilon}  \| g_\tau \|_{2},
\end{equation} 
uniformly for $\tau$ at scale $\delta$ with $R^{-\lambda}  \le \delta \leq R^{-\lambda+\varepsilon}$, then  using orthogonality, we could bound the middle term on the right-hand side of \eqref{schroform0}. Similarly, replacing the $\max_{\tau_{1}, \dots ,  \tau_{d}  \subset \tau }$ with an $\ell^2$-norm, and using the fact that there are no more than $R^\varepsilon$ choices in such a sum, in order to treat the first term it will suffice to prove
\begin{equation}\label{second0}
\Big\| \Psi_{\tau}\prod_{k=1}^{d}
|T_{\tau_{k}}g|^{\frac{1}{d}}
*\zeta_{\tau_k'} \Big\|_{L^{2}(d \mu_{R})} \lesssim 
\sqrt{c_{\alpha}(\mu)}\,
R^{\frac{\alpha}{2} - \frac{\alpha - 1}{2}
- 
 \frac{\lambda(d-\alpha)}{2(d-1)}+d\varepsilon}  \| g_\tau \|_{2},
\end{equation} 
uniformly for $\tau$ at scale $\delta$ with $R^{-\lambda}  \le \delta \leq 1$ and uniformly for choices of transversal caps 
$\tau_1,\ldots,\tau_{d}\subset \tau$. In fact we will only prove this for $\alpha> 1$ 
however we can safely ignore the other cases as Mattila  already proved the sharp bound for $\beta_d$ in low dimensions \cite{M0}. Finally, in order to deal with the remainder term, by taking $\varepsilon$ sufficiently small, it will suffice to prove that
\begin{equation}\label{thirds0}
\|\Psi_{\tau}\|_{L^{2}( d \mu_{R})} \lesssim \sqrt{c_{\alpha}(\mu)}R^{d/2+\lambda},
\end{equation}
uniformly for $\tau$ at scale $\delta$ with $R^{-\lambda}  \le \delta \leq 1$. Taking for granted the proofs of  (\ref{first0}), (\ref{second0}) and (\ref{thirds0}), which we will present in the forthcoming lemmas, starting with the easier \eqref{thirds0}, this completes the proof of Theorem~\ref{us}.

\begin{lemma}\label{third0} Let $0<\varepsilon<\frac{1}{4d}$. Then, for all  caps $\tau\sim\delta$ with $R^{-\lambda} \le \delta \leq 1$,
\begin{equation*}\label{MainLemmaFormulaPreqPreq}
\| \Psi_{\tau} \|_{L^{2}( d \mu_{R})} \lesssim \sqrt{c_{\alpha}(\mu)}R^{d/2+\lambda}.
\end{equation*}
\end{lemma}

\begin{proof} Writing $q = \frac{2(d-1)}{d-2}$, we prepare to use the property (\ref{PsiPropertyFINAL}).
 First of all, as $\Psi_{\tau}$ is essentially constant at scale one, we know that
\begin{eqnarray*}
 \int_{B_j} |\Psi_{\tau}(x)|^q d\mu_R(x)&\le& \mu_R(B_j)\sup_{x\in B_j} |\Psi_{\tau}(x)|^q\\
 &\le& c_{\alpha}(\mu_R)\sup_{x\in B_j} |\Psi_{\tau}(x)|^q \ \lesssim \  c_{\alpha}(\mu_R)\int_{B_j} |\Psi_\tau(x)|^qd x,
\end{eqnarray*}
whenever $B_j$ is a ball of diameter less than one. 
Thus, we can bound
\begin{eqnarray*}
 \| \Psi_{\tau} \|_{L^{2}( d \mu_{R})}
&\lesssim&
\mu_R(B_R)^{\frac{1}{2}-\frac{1}{q}}\| \Psi_{\tau} \|_{L^{q}(d \mu_{R})}\\
&\lesssim &c_{\alpha}(\mu_R)^{\frac{1}{2}-\frac{1}{q}}R^{\alpha(\frac{1}{2}-\frac{1}{q})}c_{\alpha}(\mu_R)^{\frac{1}{q}}\| \Psi_{\tau} \|_{L^{q}( B_{R})}
\\ 
& =&
\sqrt{c_{\alpha}(\mu)}R^{\alpha(\frac{1}{2}-\frac{1}{q})}
\| \Psi_{\tau} \|_{L^{q}(B_{R})}.
\end{eqnarray*}
Covering $B_{R}$ with a family $\{ T_{j} \}$ of  
 translates of $\tau'$ with disjoint interiors, cuboids of dimension 
$\delta^{-1} \times \dots \times \delta^{-1} \times \delta^{-2}$, we can then bound this as
\begin{eqnarray*}
\| \Psi_{\tau} \|_{L^{2}( d \mu_{R})}& \lesssim&
\sqrt{c_{\alpha}(\mu)}R^{\alpha(\frac{1}{2}-\frac{1}{q})}\Big(\sum_{j}
\| \Psi_{\tau} \|^q_{L^{q}(T_{j})}\Big)^{1/q}
\\ \nonumber
& \lesssim &
\sqrt{c_{\alpha}(\mu)}R^{\alpha(\frac{1}{2}-\frac{1}{q})}\Big(\sum_{j}
|T_{j}| |\tau '|^{\varepsilon}\Big)^{1/q}
\\ \nonumber
& \lesssim&
\sqrt{c_{\alpha}(\mu)}R^{\frac{d}{2}}\delta^{-\frac{(d+1)\varepsilon}{q}},
\end{eqnarray*}
where the second inequality is by Proposition~\ref{ThePsi}. For the range of $\delta$ under consideration, this is easily enough to give the stated bound.
\end{proof}

\begin{lemma}Let $0<\varepsilon<\frac{1}{4d}$. Then,
for all caps $\tau\sim\delta$ with $R^{-\lambda}  \le \delta \leq R^{-\lambda+\varepsilon}$,\begin{equation}\label{MainLemmaFormulaPreq}
\big\|\Psi_{\tau}| T_{\tau}g |*\zeta_{\tau'} \big\|_{L^{2}( d \mu_{R})} \lesssim \sqrt{c_{\alpha}(\mu)}R^{\frac{1}{2} - \frac{\lambda(d-\alpha)}{2(d-1)} + d\varepsilon}  \| g_\tau \|_{2}.
\end{equation} 
\end{lemma}

\begin{proof}
Again we cover $B_{R}$ by a family $\{ T_{j} \}$ of translations of $\tau'$ with disjoint interiors. 
Setting
$G_{\tau} :=  | T_{\tau}g |*\zeta_{\tau'}$, and  
denoting the measure $d\mu_{R}$ restricted to $T_{j}$ by $d\mu^j_{R}$,  
we can write
\begin{equation}\label{Ineq1}
 \| \Psi_{\tau} G_{\tau} \|_{L^{2}(d\mu_{R})}  = 
\bigg( \sum_{j} \| \Psi_{\tau} G_{\tau} \|^{2}
_{L^{2}(d\mu^j_{R})} \bigg)^{1/2}. 
\end{equation}
As in the previous lemma, we use that $\Psi_{\tau}$ is essentially 
constant at scale one, so 
\begin{eqnarray*}
\| \Psi_{\tau} \|_{L^{2}(d \mu^j_{R})}
& \lesssim &
\mu_{R}(T_{j})^{\frac{1}{2} - \frac{1}{q}}
\|  \Psi_{\tau} \|_{L^{q}(d \mu^j_{R})}\\
& \lesssim & 
\mu_{R}(T_{j})^{\frac{1}{2} - \frac{1}{q}}
c_{\alpha}(\mu_R)^{\frac{1}{q}}
\| \Psi_{\tau} \|_{L^{q}(T_{j})} \\ 
& \lesssim &
c_{\alpha}(\mu)^{\frac{1}{q}}  
\mu_{R}(T_{j})^{\frac{1}{2} - \frac{1}{q}}
R^{\frac{d+1}{2q}\varepsilon}
|T_{j}|^{\frac{1}{q}},
 %% che poi e' \\ 
%& = & R^{\varepsilon} c_{\alpha}(\mu ')^{\frac{1}{q}} \delta ^{-\frac{\alpha + 1}{d+1} %\frac{d}{2} - \frac{d-1}{2d} +
%\frac{\alpha - n}{q}}.
\end{eqnarray*}
where the final inequality is by  the property (\ref{PsiPropertyFINAL}). 
Using this and the fact that  $G_{\tau}$ is essentially constant on $T_{j}$,
\begin{eqnarray}\nonumber 
\| \Psi_{\tau} G_{\tau} \|_{L^{2}(d \mu^j_{R})}
 & \lesssim & 
 c_{\alpha}(\mu)^{\frac{1}{q}} 
\mu_{R}(T_{j})^{\frac{1}{2} - \frac{1}{q}}  R^{\frac{d+1}{2q}\varepsilon}
|T_{j}|^{\frac{1}{q}}
\| G_{\tau} \|_{L^{\infty}(T_{j})}
\\ \nonumber
 & \lesssim & 
 c_{\alpha}(\mu)^{\frac{1}{q}} 
\mu_{R}(T_{j})^{\frac{1}{2} - \frac{1}{q}} R^{\frac{d+1}{2q}\varepsilon} 
|T_{j}|^{\frac{1}{q}}
|T_{j}|^{-\frac{1}{2}} \| G_{\tau} \|_{L^{2}(T_{j})}.
\end{eqnarray} 
Plugging this into \eqref{Ineq1}, we obtain 
\begin{align} \nonumber
\|  \Psi_{\tau} G_{\tau} \|_{L^{2}( d \mu_{R})}   
& \lesssim 
c_{\alpha}(\mu)^{\frac{1}{q}}
\mu_{R} (T_{j})^{\frac{1}{2}- \frac{1}{q}} 
R^{\frac{d+1}{2q}\varepsilon}
|T_{j}|^{\frac{1}{q}} |T_{j}|^{-\frac{1}{2}}
\bigg( \sum_{j}  
\|  G_{\tau} \|^{2}_{L^{2}(T_{j})}
\bigg)^{1/2}
\\ \nonumber
& 
\lesssim  
\sqrt{c_{\alpha}(\mu)}
R^{\frac{d+1}{2q}\varepsilon} 
\delta^{(d-\alpha) \left(\frac{1}{2}-\frac{1}{q}\right)} 
\bigg( \sum_{j}  
\|  G_{\tau} \|^{2}_{L^{2}(T_{j})} \bigg)^{1/2}
\\\label{here}
& 
\lesssim 
\sqrt{c_{\alpha}(\mu)}R^{\frac{d+1}{2q}\varepsilon} 
\delta^{\frac{d-\alpha}{2(d-1)}} 
\|  G_{\tau} \|_{L^{2}(B_{R})}
\end{align}
where in the second inequality we use 
$\mu_{R}(T_{j}) \lesssim c_\alpha(\mu) \delta^{- (\alpha +1)}$ 
which follows by covering the $T_{j}$ by $\delta^{-1}$ balls of radius $\delta^{-1}$. 

On the other hand,  by Minkowski's integral inequality, we can bound
\begin{align*}
\|  G_{\tau} \|_{L^{2}(B_{R})}=\| | T_{\tau}g| * \zeta_{\tau'} \|_{L^{2}(B_{R})}
&\le
\int \| T_{\tau}g(\cdot -y)\|_{L^{2}(B_{R})}\zeta_{\tau'}(y)\, dy\\
&\le
\int \Big\|\| g_y^\vee\|_{L^{2}(\R^{d-1})}\Big\|_{L^2(|t|\le R)}\zeta_{\tau'}(y)\, dy,
\end{align*}
where 
\begin{equation}\nonumber
g_{y}(\xi) := g(\xi)\,\chi_{\tau}(\xi) e^{-i \pi(y) \cdot \xi +i(t-t_{y}) \phi(\xi)},\qquad t_y := y - \pi(y).
\end{equation}
Here $\pi$ is the orthogonal projection onto $\R^{d-1}$. Then by Plancherel's theorem, the fact that $\|g_y\|_2=\|g_{\tau}\|_2$, and the fact that the integral of $\zeta_{\tau'}$ is bounded, we obtain
$$
\|  G_{\tau} \|_{L^{2}(B_{R})}\lesssim R^{1/2}\|g_{\tau}\|_2.
$$
Plugging this into \eqref{here},  we see that
\begin{equation}\label{polka}
\big\|\Psi_{\tau}|T_{\tau}g|*\zeta_{\tau'}\big\|_{L^{2}( d \mu_{R})} \lesssim \sqrt{c_{\alpha}(\mu)}R^{\frac{1}{2} + \frac{d+2}{2q}\varepsilon} 
\delta^{\frac{d-\alpha}{2(d-1)}}  \| g_\tau \|_{2},
\end{equation}
which, with $R^{-\lambda}\le \delta \le R^{-\lambda+\varepsilon}$, yields the desired uniform estimate. \end{proof}

\begin{lemma}\label{Fuman}
Let $0<\varepsilon<\frac{1}{4d}$ and $\alpha> 1$ and $\lambda=\frac{d-\alpha}{2d-\alpha-1}$. Then,
for all caps $\tau\sim\delta$ with $R^{-\lambda}  \le \delta \leq 1$ and all $d$-transversal caps $\tau_1,\ldots\tau_{d}\sim \delta/K_{d}$ contained in~$\tau$, 
\begin{equation}\label{MainLemmaFormula}
\Big\| \Psi_{\tau}\prod_{k=1}^{d}
|T_{\tau_k}g |^{\frac{1}{d}}
*\zeta_{\tau_k'} \Big\|_{L^{2}( d \mu_{R})} 
\lesssim \sqrt{c_{\alpha}(\mu)}
R^{\frac{1}{2} -   
 \frac{\lambda(d-\alpha)}{2(d-1)}+ d\varepsilon}
\| g_\tau \|_{2}.
\end{equation} 
\end{lemma}

\begin{proof}
Setting $G_{\tau} :=  \prod_{k=1}^{d}
|T_{\tau_k}g|^{\frac{1}{d}}
*\zeta_{\tau_k'}$, we will prove that 
\begin{equation}\label{numbertwo}
\|  \Psi_{\tau}G_{\tau} \|_{L^{2}( d\mu_{R})}   
\lesssim
\sqrt{c_{\alpha}(\mu)}
R^{d\varepsilon}
R^{\frac{\alpha}{2d}} \delta^{\frac{d-\alpha}{2d(d-1)}-\frac{d-1}{2d}}
\| g_{\tau} \|_{2},
\end{equation}
which on can calculate gives the required bound for  $\delta\ge R^{-\frac{d-\alpha}{2d-\alpha-1}}$. 
By H\"older's inequality with $p=\frac{2d}{d-1}$, we first note that
\begin{eqnarray}\label{Ineq1Bis}
 \| \Psi_{\tau} G_{\tau} \|_{L^{2}(d\mu_{R})}  
 & \lesssim &
 c_{\alpha}(\mu)^{\frac{1}{2}-\frac{1}{p}} R^{\alpha \left( \frac{1}{2} - \frac{1}{p}\right)}
  \| \Psi_{\tau} G_{\tau} \|_{L^{p}(d\mu_R)}  
 \\ \nonumber
 & = &
 c_{\alpha}(\mu)^{\frac{1}{2}-\frac{1}{p}} R^{\frac{\alpha}{2d}}
\bigg( \sum_{j} \| \Psi_{\tau} G_{\tau} \|^{p}
_{L^{p}(d \mu^j_R)} \bigg)^{1/p}, 
\end{eqnarray}
where $d\mu^j_{R}$ denotes the measure $d\mu_{R}$ a member of the cover of $B_R$ by translates of $\tau'$. Using that $\Psi_{\tau}$ is essentially 
constant at scale one,
\begin{eqnarray*}
\| \Psi_{\tau} \|_{L^{p}(d\mu^j_R)}
& \lesssim &
\mu(T_{j})^{\frac{1}{p} - \frac{1}{q}}
\| \Psi_{\tau} \|_{L^{q}(d \mu^j_R)} \\ 
& \lesssim &
\mu(T_{j})^{\frac{1}{p} - \frac{1}{q}}c_{\alpha}(\mu)^{\frac{1}{q}}
\| \Psi_{\tau} \|_{L^{q}(T_{j})} \\ 
& \lesssim &
\mu(T_{j})^{\frac{1}{p} - \frac{1}{q}}c_{\alpha}(\mu)^{\frac{1}{q}}
|T_{j}|^{\frac{1}{q}}R^{\frac{d+1}{2q}\varepsilon},
\end{eqnarray*}
where the final inequality is by the property (\ref{PsiPropertyFINAL}). 
As $\tau'\subset \tau_k'$ we still have that that  $G_{\tau}$ is essentially constant on $T_{j}$, so that
\begin{eqnarray}\nonumber 
\| \Psi_{\tau} G_{\tau} \|_{L^{p}(d\mu^j_R)}
 & \lesssim & 
\mu(T_{j})^{\frac{1}{p} - \frac{1}{q}}c_{\alpha}(\mu)^{\frac{1}{q}}
|T_{j}|^{\frac{1}{q}}R^{\frac{d+1}{2q}\varepsilon}
\| G_{\tau} \|_{L^{\infty}(T_{j})}
\\ \nonumber
 & \lesssim & 
\mu(T_{j})^{\frac{1}{p} - \frac{1}{q}}c_{\alpha}(\mu)^{\frac{1}{q}}
R^{\frac{d+1}{4}\varepsilon}
|T_{j}|^{\frac{1}{q} - \frac{1}{p}}
\| G_{\tau} \|_{L^{p}(T_{j})}
\\ \nonumber
& \lesssim &
c_{\alpha}(\mu)^{\frac{1}{p}}
R^{\frac{d+1}{4}\varepsilon}
\delta^{(d-\alpha) \left( \frac{1}{p} - \frac{1}{q} \right)}
\| G_{\tau} \|_{L^{p}(T_{j})}
\\ \nonumber
& = &
c_{\alpha}(\mu)^{\frac{1}{p}}
R^{\frac{d+1}{4}\varepsilon}
\delta^{\frac{d-\alpha}{2d(d-1)}}
\| G_{\tau} \|_{L^{p}(T_{j})}
\end{eqnarray} 
Plugging this into \eqref{Ineq1Bis}, we obtain 
\begin{align*} 
\|  \Psi_{\tau} G_{\tau} \|_{L^{2}( d \mu_{R})}   
& \lesssim 
\sqrt{c_{\alpha}(\mu)}
R^{\frac{d+1}{4}\varepsilon} 
R^{\frac{\alpha}{2d}}
\delta^{\frac{d-\alpha}{2d(d-1)}}
\|  G_{\tau} \|_{L^{p}(B_R)}.
\end{align*}

In order to bound $\|G_{\tau}\|_{L^{p}(B_{R})}$, we write 
\begin{eqnarray}\nonumber
G_{\tau}(x,t) & = &
\prod_{k=1}^{d} \int |T_{\tau_{k}} g|^{\frac{1}{d}}
((x,t)-y_{k}) \zeta_{\tau_k'} (y_{k}) \ d y_{k}
\\ \nonumber
& = &
\prod_{k=1}^{d} \int |T_{\tau_{k}} g_{y_{k}}|^{\frac{1}{d}}
(x,t) \zeta_{\tau_k'} (y_{k}) \ d y_{k},
\end{eqnarray}
where this time
\begin{equation}\nonumber
g_{y_{k}} := g\,\chi_{\tau_k} e^{-i \pi(y_{k}) \cdot \xi -it_{k} \phi(\xi)},\qquad t_k := y_{k} - \pi(y_{k}).
\end{equation}
Then, by Minkowski's integral inequality, 
it will suffice to bound
$$
 \int\Big\|\prod_{k=1}^{d} |T_{\tau_k}g_{y_k}|^{\frac{1}{d}}
\Big\|_{L^{p}(B_{R})} \prod_{k=1}^{d}\zeta_{\tau_k'}(y_k)\, d y_{1}\ldots dy_{d}.
$$
Again $\|g_{y_k}\|_2=\|g_{\tau_k}\|_2$, and so it remains to prove the multilinear extension estimate
\begin{equation}\label{rrr}\Big\|\prod_{k=1}^{d} |T_{\tau_k}g|^{\frac{1}{d}}
\Big\|_{L^{p}(B_{R})}\lesssim R^{\varepsilon}\delta^{-\frac{d-1}{2d}}\| g_\tau \|_{2}.
\end{equation}
We recall that
$\tau_{k}$ are traversal caps at scale $\delta/K_d$  and so a direct application of Theorem~\ref{MultilinEstimates} would give us the inequality with the constant $\mathfrak{c}(\delta^d K^{-d}_d)$. We do not know how large this is, however we have chosen the scales so that at least we know that  $\mathfrak{c}(K^{-d}_d)\le R^{\varepsilon^2}$. Thus, using the fact the caps $\tau_k$ are contained in  $\tau$ at scale $\delta$, we can first modulate and scale the inequality in order to get into this situation.

Denoting by $\xi_{0}$ the center of $\pi(\tau)={Q}$ we let $\widetilde{Q}_k$ be the scaled versions of $Q_k$ which are first translated by $-\xi_{0}$. 
Indeed, introducing new variables, 
\begin{equation*}
(x', t') = (\delta x, \delta^{2} t), \quad
%(\pi(y_{k})', y_{k, d}') = (\delta \pi(y_{k}), \delta^{2} y_{k, d}), \quad
\xi - \xi_{0} = \delta \xi', 
\end{equation*}
and writing 
\begin{equation*}\label{Def:gTauBar}
f(\xi') :=
\delta^{\frac{d-1}{2}}g(\xi_{0}+\delta \xi'),
\end{equation*}  so that
$\| f \|_{2} 
= \|g\|_{2}$,
it is trivial to calculate that
\begin{equation*}\nonumber
|T_{\tau_k}g (x,t)|  =  \delta^{\frac{d-1}{2}}|\widetilde{T}_{\widetilde{\tau}_k}f(x' + \delta^{-1} \nabla\phi (\xi_0) t', t')|,
\end{equation*}
where
\begin{equation}\label{polk}
\widetilde{T}_{\widetilde{\tau}_k}f(x, t):=\int_{\widetilde{Q}_k} 
e^{i x \cdot \xi + i tS_{\xi_0,\delta}\phi(\xi)}
f(\xi) 
\ d \xi
\end{equation}
and the scaled phase is given by
$$
S_{\xi_0,\delta}\phi(\xi')=\delta^{-2}\Big(\phi(\xi_0+\delta\xi')-\delta\nabla \phi(\xi_0)\cdot \xi'-\phi(\xi_0)\Big).
$$
The $d$-transversal caps $\widetilde{\tau}_{k}$ satisfying $\pi(\widetilde{\tau}_{k})=\widetilde{Q}_k$ 
are now at scale $1/K_{d}$. Writing
\begin{align*}
\prod_{k=1}^{d} |T_{\tau_k}g|^{\frac{1}{d}}(x,t)
& =  
\delta^{\frac{d-1}{2}} \prod_{k=1}^{d}
 |\widetilde{T}_{\widetilde{\tau}_k}f|^{\frac{1}{d}}
(x' + \delta^{-1} \nabla\phi (\xi_0) t', t'),
\end{align*}
we see that the left-hand side of \eqref{rrr} is bounded by
\begin{align*}   
\delta^{ \frac{d-1}{2} - \frac{d+1}{p} }\Big(\int_{|t'|\le \delta^2R}\int_{|x'|\le \delta R}\Big|\prod_{k=1}^{d}
 |\widetilde{T}_{\widetilde{\tau}_k}f|^{\frac{1}{d}}&
(x' + \delta^{-1} \nabla\phi (\xi_0) t', t')\Big|^p dx'dt'\Big)^{1/p}\\
\le \,
\delta^{-\frac{d-1}{2d}} \Big\|\prod_{k=1}^{d}
 |\widetilde{T}_{\widetilde{\tau}_k}f|^{\frac{1}{d}}\Big\|_{L^{p}([-\delta^2R,\delta^2R]\times B_{2\delta R})},
\end{align*}
Here, we change variables $x=x' + \delta^{-1} \nabla\phi(\xi_{0})t'$ and use that $\delta^{-1} \nabla\phi(\xi_{0})t'$ is bounded above by $\delta R$ so that the oblique tube can be covered by the fatter cylinder. Now, by Proposition~\ref{MultilinEstimates2},
$$
\Big\| \prod_{k=1}^{d}
 |\widetilde{T}_{\widetilde{\tau}_k}f|^{\frac{1}{d}}\Big\|_{L^{p}(\R^{d-1}\times[-\delta^2R,\delta^2R])}\lesssim \mathfrak{c}(\varepsilon)(\delta^2R)^\varepsilon \|f\|_2,
$$
and so altogether we get  \eqref{rrr}, which completes the proof.
\end{proof}

The conjectured $m$-linear extension  estimates \cite[Conjecture~4]{B}, with $m\le d-1$, combined with the arguments of this section, would yield 
\begin{equation}\label{fre}
\beta_d(\alpha)\ge \min \Big\{ \alpha-1+\frac{(d-\alpha)(d+m-2\alpha)}{2(m-1)(d+m-\alpha-1)},\alpha-\frac{2\alpha}{d+m}\Big\},
\end{equation}
whenever $3\le m\le d-1$. Comparing the second term in the minimum with the bound of Theorem~\ref{us}, it is clear that this is not an improvement for larger $\alpha$. However, by taking $m=d/2+1$ (assuming that $d$ is even), this would improve our bound and  Erdo\u{g}an's in a neighbourhood of $\alpha=d/2+2/3$. It would not  be sufficient to improve the state-of-the-art for Falconer's conjecture however. Using the partial results for $m$-linear restriction already proven in \cite[formula (40)]{BCT}, by the same argument one obtains
\begin{equation*}\label{free}
\beta_d(\alpha)\ge \min \Big\{ \alpha-1+\frac{(d-\alpha)(m-\alpha)}{(m-1)(2m-\alpha-1)},\alpha-\frac{\alpha}{m}\Big\}.
\end{equation*}
%whenever $3\le m\le d-1$.

Given that $m$-linear estimates necessarily have worse integrability properties than the $d$-linear estimates of Section~\ref{mult}, it is not obvious that anything can be gained by using these.  The reason that they can be effective is that the decomposition of Bourgain and Guth  improves if we take the initial dichotomy at a lower level of multilinearity. The improvement manifests itself in the fact that the functions $\Psi_\tau$ have better integrability properties and so we pay less while removing them. This kind of thing was first observed by Temur  in the context of the linear restriction problem  \cite{Tem}. Here, the reduced integrability in the estimates leads to  the estimate \eqref{numbertwo}
 having a  worse dependency on $R$ (this produces the second term in the minimum), however the improved properties of $\Psi_\tau$ lead to both \eqref{numbertwo} and \eqref{polka} having a better dependency on $\delta$, and together they would yield \eqref{fre}  after choosing the limiting scale $\lambda$  in an optimal fashion.

\section{Proof of Proposition~\ref{propo}}\label{prop}

In order to avoid repetition in the following section, we consider $m\ge 1$, however it will suffice to consider $m=1$ here. If $v_{0}$ and $v_1$ are in the Schwartz class then the solution $v$ to the wave equation with this initial data can be written as 
\begin{align*}
v(\cdot,t)&=\cos(t\sqrt{-\Delta}) v_0+\frac{\sin(t\sqrt{-\Delta})}{\sqrt{-\Delta}} v_1\\
&=e^{it\sqrt{-\Delta}}f_++e^{-it\sqrt{-\Delta}}f_-.
\end{align*}
Here $f_+=\frac{1}{2}(v_0-iI_1\ast v_1)$ and $f_-=\frac{1}{2}(v_0+iI_1\ast v_1)$, where $I_1$ is the Riesz kernel, and
\begin{equation*}
e^{i t (-\Delta)^{m/2}} \!f(x):=\frac{1}{(2 \pi)^{d/2}}
\int_{\R^d}\widehat{f}(\xi)\, e^{ix\cdot\xi +it | \xi |^{m}} d \xi. 
\end{equation*}
For data in $\dot{H}^{s}\times\dot{H}^{s-1}$, both $f_+$ and $f_-$ belong to $\dot{H}^s$, however  this integral
does not necessarily exist in the sense of Lebesgue for $s\le n/2$.
Instead we define $v(x, t)$ to be the pointwise limit
\begin{equation}\label{UPointDef0}
v(x, t) := \lim_{N \rightarrow \infty} 
S^{\!N,1}_{t} f_{+}(x)+S^{\!N,1}_{-t} f_{-}(x),  
\end{equation}
whenever the limit exists, where
\begin{equation*}\label{SNTDef}
S^{\!N,m}_{t} f :=
\int_{\R^d} \psi\left( \frac{| \xi |}{N}\right) 
\widehat{f}(\xi)\, e^{i x \cdot \xi +it | \xi |^{m}}  d \xi
\end{equation*}
and $\psi$ is a positive Schwartz function that equals $(2\pi)^{-d/2}$ at the origin.
This
coincides almost everywhere with the classical  
 solution defined via the $L^{2}$-limit. 

Writing $\|I_{s} * f\|_{\dot{H}^s}:=\|f\|_2$, we know that $f_+$, $f_-$ and the limit (\ref{UPointDef0}) are well-defined with 
respect to fractal measures
provided that $\alpha > d-2s$ due to the inequalities 
\begin{equation*}\label{SimpleMaxIneqPrelim}
\| I_{s} * f\|_{L^{1}(d \mu)}
\lesssim \sqrt{c_{\alpha}(\mu)\|\mu\|}\, \| f \|_{2},
\end{equation*}
\begin{equation*}\label{SimpleMaxIneq}
\Big\| \sup_{N > 1} | S^{\!N,m}_{t} I_{s} * f | \Big\|_{L^{1}(d \mu)}
\lesssim \sqrt{c_{\alpha}(\mu)\|\mu\|}\, \| f \|_{2};
\end{equation*}
 see for example \cite{BBCR}, \cite{BR} or \cite[Chapter 17]{M}.
Then by standard arguments (see for example Appendix B of~\cite{BR}) and an application of Frostman's lemma (see for example \cite[Theorem 2.7]{M}), 
the implication
$$
\beta_d(\alpha)> d-2s\quad \Rightarrow \quad  \gamma_{d}(s)\le \alpha
$$
 can be deduced from from the following lemma.

\begin{lemma}\label{loc} Let $m\ge1$, $d\ge 2$ and
 $0<s<d/2$. Then 
\begin{equation*}
\Big\|\sup_{t\in\R}\sup_{N\ge1}|S^{\!N,m}_{t}I_{s} * f|\Big\|_{L^1(d\mu)}\lesssim
\sqrt{c_\alpha(\mu)\|\mu\|}\|f\|_{2}
\end{equation*}
whenever $f\in L^2(\R^d)$, $\mu$ is an $\alpha$-dimensional measure and $s>\frac{d-\beta_d(\alpha)}{2}$. 
\end{lemma}
\begin{proof} First of all we remark that the maximal function is Borel measurable by comparing with the maximum function with time restricted to the rationals; see \cite[Lemma 17.7]{M}. Then, using polar coordinates we write
\begin{align*}
|S^{\!N,m}_{t}I_{s} * f(x)|&=\left|\int_{\R^d}\psi(N^{-1}|\xi|)\,|\xi|^{-s}\widehat{f}(\xi)\,e^{i(x\cdot\xi + t|\xi|^m)} \, d\xi\right|\\
&= \left|\int_0^\infty \psi(N^{-1}R)\,R^{d-1-s}e^{itR^m}\!\!\int_{\mathbb{S}^{d-1}}\widehat{f}(R\omega)\, e^{iRx\cdot\omega}d\sigma(\omega)\,dR\right|\\
& \lesssim \int_0^\infty
R^{d-1-s}\left|\int_{\mathbb{S}^{d-1}}\widehat{f}(R\omega)
\,e^{iRx\cdot\omega}d\sigma(\omega)\right|dR,
\end{align*}
so that, by Fubini's theorem,
\begin{equation}\label{til}
\Big\|\sup_{t\in\R}\sup_{N\ge1}|S^{\!N}_{t}I_{s} * f|\Big\|_{L^1(d\mu)}\lesssim
\int_0^\infty
\!\!\!\!\!R^{d-1-s}\big\|\big(\widehat{f}(R\,\cdot)d\sigma\big)^\vee(R\,\cdot\,)\big\|_{L^1(d\mu)}dR.
\end{equation}
Noting that, even when $R$ is small, we have 
\begin{equation*}
\|\widehat{\mu}(R\,\cdot\,)\|^2_{L^2(\mathbb{S}^{d-1})}\lesssim \|\mu\|^2\lesssim
c_{\alpha}(\mu)\|\mu\|,\end{equation*}
 the inequality \eqref{dz}
implies by duality that
$$
\big\|\big(\widehat{f}(R\,\cdot)d\sigma\big)^\vee(R\,\cdot\,)\big\|_{L^1(d\mu)}\lesssim
\sqrt{c_\alpha(\mu)\|\mu\|}\,(1+R)^{-\beta/2}\|\widehat{f}(R\,\cdot\,)\|_{L^2(\mathbb{S}^{d-1})}.
$$
for all $\beta<\beta_d(\alpha)$, so that (\ref{til}) is bounded by
\begin{align*}
&\!\!\!\!\!\!\!\!\lesssim \sqrt{c_\alpha(\mu)\|\mu\|}\int_0^\infty
\!\frac{R^{d-1-s}}{(1+R)^{\beta/2}}\|\widehat{f}(R\,\cdot\,)\|_{L^2(\mathbb{S}^{d-1})}dR.
\end{align*}
Finally, by an application of the Cauchy--Schwarz inequality, we can continue to estimate as
\begin{align*}
&\lesssim \sqrt{c_\alpha(\mu)\|\mu\|}\left(\int_0^\infty \frac{R^{d-1-2s}}{(1+R)^{\beta}}dR\right)^{1/2}\left(\int_0^\infty\|\widehat{f}(R\,\cdot\,)\|^2_{L^2(\mathbb{S}^{d-1})}R^{d-1}dR\right)^{1/2}\\
&\lesssim \sqrt{c_\alpha(\mu)\|\mu\|}\,\|f\|_{L^2(\R^d)},
\end{align*}
where in the final inequality we choose $\beta$ so that $\beta_d(\alpha)>\beta>d-2s$ as we may.\end{proof}

\section{Proof of Theorem~\ref{bil}}
\label{schrodingerconv}

As in the previous section, if $i\partial_t u+\Delta u=0$ and the initial data $u_{0}$ is in the Schwartz class, we can write 
\begin{equation*}
u(x, t) = e^{i t \Delta}u_{0}(x):=\frac{1}{(2 \pi)^{n/2}}
\int_{\Rn}\widehat{u}_{0}(\xi)\, e^{ix\cdot\xi -it | \xi |^{2}} d \xi , 
\end{equation*}
however for data in $H^{s}$  we define
\begin{equation}\label{UPointDef}
u(x, t) := \lim_{N \rightarrow \infty} 
S^{\!N,2}_{-t} u_{0}(x)  
\end{equation}
whenever the limit exists. This 
coincides almost everywhere with the classical  
 solution defined via the $L^{2}$-limit. Then, by standard arguments,
an upper bound for $\alpha_{n}(s)$ can be obtained from 
appropriate maximal inequalities with respect to fractal measures. We summarise this in the following lemma.

\begin{lemma}\label{BridgeLemma} \cite{BBCR}
Let $\alpha > \alpha_{0} \geq n-2s$ and suppose that
\begin{equation*}\label{CompleteMaxIneq}
\Big\| \sup_{0<t<1}  | e^{i t \Delta}u_{0} | \Big\|_{L^{1}(d \mu)} \lesssim
\sqrt{c_{\alpha}(\mu)\|\mu\|}\, \| u_{0} \|_{H^{s}(\Rn)}
\end{equation*}
whenever $u_0$ is in the  Schwartz class and $\mu$ is an $\alpha$-dimensional. 
Then $\alpha_{n}(s) \leq \alpha_{0}$.
\end{lemma}

\begin{proof} First we use the argument at the beginning of the proof of Proposition 3.2 in \cite{BBCR} to conclude that \eqref{CompleteMaxIneq} implies the maximal estimate
$$\Big\| \sup_{0<t<1}\sup_{N > 1} | S^{\!N,2}_{-t} u_{0} | \Big\|_{L^{1}(d \mu)}
\lesssim \sqrt{c_{\alpha}(\mu)\|\mu\|}\, \| u_{0} \|_{H^{s+\varepsilon}(\Rn)}
$$
whenever $u_0\in H^{s+\varepsilon}$ for all $\varepsilon>0$. Then we use the density argument that invokes Frostman's lemma  in the Appendix B of \cite{BR} or \cite[Chapter 17]{M} to conclude.
\end{proof}

Thus it remains to prove {\it a priori} maximal estimates that hold uniformly with respect to compactly supported fractal measures. Indeed it remains to prove the following theorem. 
\begin{theorem}\label{OurBouThm}
Let  $n\ge 1$ and 
\begin{equation*} \label{sValue}
s > \left\{
\begin{array}{lcr}
\frac{n-\alpha}{2} + \frac{n}{2(n+1)}, & \mbox{if} & 
0 \leq \alpha \leq  n - 1+\frac{2}{n+1},  \\
 &  & \\
(n - \alpha +1) \left( \frac{1}{2} - \frac{1}{4n} \right)  & \mbox{if} &  
n - 1+\frac{2}{n+1} \leq \alpha \leq n.
\end{array} \right.
\end{equation*}
Then
\begin{equation*}\label{AMIThm}
\left\| \sup_{0< t < 1} | e^{it\Delta} f | \right\|_{L^{2}(d \mu)} \lesssim
\sqrt{c_{\alpha}(\mu)}
\| f \|_{H^{s}(\Rn)}
\end{equation*}
whenever $f$ is Schwartz and $\mu$ is $\alpha$-dimensional.
\end{theorem}

The result, although true with $n=1$, is already bettered by the work of \cite{BBCR}. 
This extends to fractal measures the following theorem due to Bourgain (with $n=1$ due to Carleson \cite{Carl} and with $n=2$ due to Lee \cite{Lee}).

\begin{theorem} \cite{Bou2}
Let $n\ge 1$ and  $s > \frac{1}{2} - \frac{1}{4n}$. Then 
\begin{equation*}\label{BouMaxIn}
\Big\| \sup_{0 < t < 1} | e^{it \Delta} f | \Big\|_{L^{2}(B_1)} 
\lesssim \| f \|_{H^{s}(\Rn)}.
\end{equation*} 
\end{theorem}

\begin{proof}[Proof of Theorem~\ref{OurBouThm}]

Set $s_{\mathrm{o}}= \max 
\{ \frac{n-\alpha}{2} + \frac{n}{2(n+1)}, (n - \alpha +1) ( \frac{1}{2} - \frac{1}{4n} ) \}$. After noting that
\begin{eqnarray}\nonumber
\left\| \sup_{0 < t < 1 } |e^{it\Delta} f| \right\|_{L^{2}(d \mu)}
& \lesssim &
\sqrt{c_{\alpha}(\mu)}\, 
\left\| \sup_{0 < t < 1 } |e^{it\Delta} f| \right\|_{L^{\infty}(B_1)}
\\ \nonumber
& \lesssim & \sqrt{c_{\alpha}(\mu)} |B_{2^{\mathfrak{c}(\varepsilon)}}|^{1/2}  
 \big\| \widehat{f}\, \big\|_{2}
\\ \nonumber
&\lesssim & 
\sqrt{c_{\alpha}(\mu)} \| f \|_{2}
\end{eqnarray}
provided $\mbox{supp}\widehat{f} \subset \{\xi\in\R^n\,:\, | \xi | \le 2^{\mathfrak{c}(\varepsilon)} \}$, by a  dyadic decomposition in frequency, 
the inequality (\ref{AMIThm}) 
would follow from 
\begin{equation*}\label{AlphaMaxIneq}
\left\| \sup_{0 < t <1} | e^{it\Delta} f | \right\|_{L^{2}(d \mu)} \lesssim
\sqrt{c_{\alpha}(\mu)} 
R^{s_{\mathrm{o}}+ \varepsilon} \| f\|_{2},
\end{equation*}
provided 
$\mbox{supp}\widehat{f} \subset \left\{ R/8 < | \xi | < R/2 \right\}$ for all   $R > 2^{\mathfrak{c}(\varepsilon)}$. 
For this we make use of temporal localisation lemma due to Lee \cite[Lemma 2.3]{Lee}. In fact we use a version that holds with respect to fractal measures and where the $\varepsilon$-loss in derivatives was avoided (see \cite[Lemma 2.1]{KeithLee}), so that it will suffice to prove  
 \begin{equation*}\label{RestrMaxIneq}
\bigg\| \sup_{0 < t < 1/R} | e^{it\Delta} f | \bigg\|_{L^{2}(d \mu)} 
\lesssim
\sqrt{c_{\alpha}(\mu)}R^{s_{\mathrm{o}}+ \varepsilon} \| f \|_{2}.
\end{equation*}
Writing $\widehat{f}_R=R^n\widehat{f}(R\, \cdot\,)$ and scaling,  we see that
\begin{eqnarray*}
\bigg\| \sup_{0 < t < 1/R} | e^{it\Delta} f | \bigg\|_{L^{2}(d \mu)}&= &R^{-\alpha/2}\bigg(
\int \sup_{0 < t < R} |e^{it\Delta} f_R|^2 (x) R^{\alpha} d \mu (x / R)
\bigg)^{1/2} \\ \nonumber
\end{eqnarray*}
so that, by writing $d \mu_{R}(x) := R^{\alpha} d \mu (x/R)$, this is equivalent to 
\begin{equation*}\label{SecondInTheEquiv}
\Big\| \sup_{0 < t < R} | e^{it\Delta} f | \Big\|_{L^{2}(d \mu_{R})} \lesssim \sqrt{c_{\alpha}(\mu)}R^{\frac{\alpha - n}{2} + s_{\mathrm{o}}+ \varepsilon}  \| f \|_{2},
 \end{equation*}
provided $\supp \widehat{f} \subset \{\xi\,:\, 1/8\le|\xi|\le 1/2 \}$. It is easy to check that $c_\alpha(\mu_R)=c_\alpha(\mu)$. 
%Note that $\mu_{R}$ is supported on a ball of radius $Rr$, however we can suppose from now on that $r=1$ by a finite covering.%\begin{remark}
%It is easy to check that $\mu_{R}$ is
%a positive $\alpha$-dimensional Borel 
%measure supported $B_{R} \subset \mathbb{R}^{n}$ and that 
%$c_{\alpha}(\mu_{R}) = c_{\alpha}(\mu)$. \texttt{I don't think this is true}
%\end{remark}

Now by taking $\lambda=1/2$ in \eqref{FinalSelfIteratedFormula} we have the pointwise bound
\begin{eqnarray*}\nonumber
 | e^{it\Delta} f |    
& \lesssim &  
R^{\varepsilon}\!\!\!\!\sum_{ 
R^{-1/2} \le \delta \leq 1 }
\Bigg( 
\sum_{ \tau \sim \delta} 
\Big( \Psi_{\tau}\!\!\!\!
\max_{\tau_{1}, \dots ,  \tau_{n+1}  \subset \tau }\prod_{k=1}^{n+1}
|T_{\tau_{k}} \widehat{f}\,|^{\frac{1}{n+1}}
*\zeta_{\tau_k'} \Big)^{2}   
\Bigg)^{1/2}  
\\   \nonumber
& &+
\sum_{R^{-1/2}  \le \delta \leq R^{-1/2+\varepsilon} } 
\left( \sum_{\tau \sim \delta} 
\left(\Psi_{\tau} |T_{\tau} \widehat{f}\,|*\zeta_{\tau'}  \right)^{2}  
\right)^{1/2}
\\ \label{schroform}
&  &+
\sum_{R^{-1/2} \le \delta \leq 1}  
\left( \sum_{\tau \sim \delta}\Psi^2_{\tau}\right)^{1/2}R^{-1/\varepsilon}\|\widehat{f}\,\|_{2}.
\end{eqnarray*}
Recalling that there are a finite number, independent of $R$, of terms in each of the $\delta$-sums, by the triangle inequality, we need only prove estimates which are uniform in $\delta$. Writing $g_\tau:=\widehat{f} \chi_\tau$, if we could prove
\begin{equation}\label{first}
\Big\| \sup_{0 < t < R} \Psi_{\tau}|T_{\tau} g|*\zeta_{\tau'}  \Big\|_{L^{2}(d \mu_{R})} \lesssim \sqrt{c_{\alpha}(\mu)}R^{\frac{\alpha - n}{2} + s_{\mathrm{o}}+ n\varepsilon}  \| g_\tau \|_{2},
\end{equation} 
uniformly for $\tau$ at scale $\delta$ with $R^{-1/2}  \le \delta \leq R^{-1/2+\varepsilon}$, then  using orthogonality, we could bound the middle term on the right-hand side of \eqref{schroform}. Similarly, replacing the $\max_{\tau_{1}, \dots ,  \tau_{n+1}  \subset \tau }$ with a $\ell^2$-norm, and using the fact that there are no more than $R^\varepsilon$ choices in such a sum, in order to treat the first term it will suffice to prove
\begin{equation}\label{second}
\Big\| \sup_{0 < t < R} \Psi_{\tau}\prod_{k=1}^{n+1}
|T_{\tau_{k}} g|^{\frac{1}{n+1}}
*\zeta_{\tau_k'} \Big\|_{L^{2}(d \mu_{R})} \lesssim \sqrt{c_{\alpha}(\mu)}\,R^{\frac{\alpha - n}{2} + s_{\mathrm{o}}+ n\varepsilon}  \| g_\tau \|_{2},
\end{equation} 
uniformly for $\tau$ at scale $\delta$ with $R^{-1/2}  \le \delta \leq 1$ and uniformly for choices of transversal caps $\tau_1,\ldots,\tau_{n+1}\subset \tau$. Finally, in order to deal with the remainder term,  by taking $\varepsilon$ sufficiently small, it will suffice to prove that
\begin{equation}\label{thirds}
\|\sup_{0<t<R} \Psi_{\tau} \|_{L^{2}( d \mu_{R})} \lesssim \sqrt{c_{\alpha}(\mu)}R^{n+1},
\end{equation}
uniformly for $\tau$ at scale $\delta$ with $R^{-1/2}  \le \delta \leq 1$. Taking for granted the proofs of  (\ref{first}), (\ref{second}) and (\ref{thirds}), which we will present in the forthcoming lemmas, this completes the proof of Theorem~\ref{OurBouThm}. 
\end{proof}

From now on, for nested norms, we write
$
\| f \|_{XY} := \big\|\|f\|_{Y}\big\|_{X}.
$

\begin{lemma}\label{third} Let $0<\varepsilon<\frac{1}{8n}$. Then, for all  caps $\tau\sim\delta$ with $R^{-1/2} \le \delta \leq 1$,
\begin{equation*}
\| \Psi_{\tau} \|_{L^{2}( d \mu_{R}) L^{\infty}(0,R)} \lesssim \sqrt{c_{\alpha}(\mu)}R^{n+1}.
\end{equation*}
\end{lemma}

\begin{proof} Writing $q=\frac{2n}{n-1}$, we prepare to apply Proposition~\ref{ThePsi}.
 First of all, as $\Psi_{\tau}$ is essentially constant at scale one,
we can bound
\begin{eqnarray*}
 \| \Psi_{\tau} \|_{L^{2}( d \mu_{R}) L^{\infty}(0,R)}
&\lesssim &
\| \Psi_{\tau} \|_{L^{2}( d \mu_{R}) L^{q}(0,R)}
\\ 
&\lesssim &
\sqrt{c_{\alpha}(\mu)}\| \Psi_{\tau} \|_{L^{2}( B_{R}) L^{q}(0,R)}
\\ 
& \lesssim & 
\sqrt{c_{\alpha}(\mu)}R^{n(\frac{1}{2}-\frac{1}{q})}
\| \Psi_{\tau} \|_{L^{q}(B_{R}\times(0,R))}.
\end{eqnarray*}
Noting that $n(\frac{1}{2}-\frac{1}{q})=\frac{1}{2}$, and covering $B_{R} \times (0,R)$ with a family $\{ T_{j} \}$ of  
 translates of $\tau'$ with disjoint interiors, we can bound this as
\begin{eqnarray*}\label{ChainOfIneqTri}
\| \Psi_{\tau} \|_{L^{2}( d \mu_{R}) L^{q}(0,R)}& \lesssim &
\sqrt{c_{\alpha}(\mu)}R^{\frac{1}{2}}\Big(\sum_{j}
\| \Psi_{\tau} \|^q_{L^{q}(T_{j})}\Big)^{1/q}
\\ \nonumber
& \lesssim &
\sqrt{c_{\alpha}(\mu)}R^{\frac{1}{2}}\Big(\sum_{j}
|T_{j}| |\tau '|^{\varepsilon}\Big)^{1/q}
\\ \nonumber
& \lesssim &
\sqrt{c_{\alpha}(\mu)}R^{\frac{1}{2}}R^\frac{n+1}{q}\delta^{-\frac{(n+2)\varepsilon}{q}},
\end{eqnarray*}
where the second inequality is by Proposition~\ref{ThePsi}. For the range of $\delta$ under consideration, this is more than enough to give the desired bound.
\end{proof}

\begin{lemma}Let $0<\varepsilon<\frac{1}{8n}$. Then,
for all caps $\tau\sim\delta$ with $R^{-1/2}  \le \delta \leq R^{-1/2+\varepsilon}$,\begin{equation*}
\big\|\Psi_{\tau}|T_{\tau} g|*\zeta_{\tau'}\big\|_{L^{2}( d \mu_{R}) L^{\infty}(0,R)} \lesssim \sqrt{c_{\alpha}(\mu)}R^{\frac{1}{2}-\frac{1}{4n}-\frac{n-\alpha}{4n} + n\varepsilon}  \| g_\tau \|_{2}.
\end{equation*} 
\end{lemma}

\begin{proof}
We cover $B_{R} \times (0,R) $ by a family $\{ T_{jk} \}$ of translations of $\tau'$ with disjoint interiors. 
Denote by $I_{j}$ the projection orthogonal to time of $T_{jk}$ onto $\R^n$.
Recall that the sets $T_{jk}$ have dimensions 
$\delta^{-1} \times \dots \times \delta^{-1} \times \delta^{-2} $ and, as our functions are frequency supported in the the unit annulus, the sets $\tau'$ make an angle greater than $\pi/8$ with the time axis. Thus
the projections $I_{j}$ also have a long side of length a constant multiple of $\delta^{-2}.$

Set
$G_{\tau} :=  |T_{\tau} g|*\zeta_{\tau'}$. 
Denoting by $d\mu^j_{R}$ the measure $d\mu_{R}$ restricted to $I_j$,  by H\"older's inequality 
\begin{eqnarray*}\label{Ineq1}
 \| \Psi_{\tau} G_{\tau} \|_{L^{2}(d\mu_{R}) L^{\infty}(0,R)}&  = &
\bigg( \sum_{j} \| \Psi_{\tau} G_{\tau} \|^{2}
_{L^{2}(d\mu^j_{R})L^{\infty}(0,R)} \bigg)^{1/2} 
\\ \nonumber
& \leq & \mu_{R}( I_{j})^{\frac{1}{2} - \frac{1}{p}} 
\bigg( \sum_{j} \| \Psi_{\tau} G_{\tau} \|^{2}_{L^{p}(d\mu^j_{R})L^{\infty}(0,R)} \bigg)^{1/2}
%che poi e' \\ \nonumber
%& = & \delta^{-\frac{\alpha +1}{2(n+1)}} 
%\left( \sum_{j} \| \Psi_{\tau} F_{\tau} \|^{2}_{L^{p}_{I_{j}, d\mu_{R}}L^{\infty}_{|x_{n+1}| %\leq R}} \right)^{1/2}. 
\end{eqnarray*}
Denoting $T_{jk}^x=\{(y,t)\in T_{jk}\,:\, y=x\}$, on the other hand we have 
\begin{equation}\nonumber
\sup_{0<t<R } | \Psi_{\tau} G_{\tau}|(x,t) \leq 
\bigg( \sum_{k} \| \Psi_{\tau} G_{\tau} \|^{p}_{L^{\infty}(B^x_{jk})} \bigg)^{1/p}, 
\end{equation}
for all $x  \in I_{j}$, so that
\begin{align}\label{I}
\| \Psi_{\tau}G_{\tau} &\|_{L^{2}( d\mu_{R}) L^{\infty}(0,R)}\\\nonumber & \leq 
\mu_{R}( I_{j})^{\frac{1}{2} - \frac{1}{p}}
% che poi e' \delta^{-\frac{\alpha +1}{2(n+1)}}
\bigg( \sum_{j} \bigg( \sum_{k} 
\| \Psi_{\tau} G_{\tau} \|^{p}_{L^{p}L^{\infty}(T_{jk}, d \mu_{R}dt)}
\bigg)^{2/p} \bigg)^{1/2}\!\!\!\!.
\end{align}

As in the previous lemma, we use that $\Psi_{\tau}$ is essentially constant at scale one, so that
\begin{eqnarray*}
\| \Psi_{\tau} \|_{L^{p}L^{\infty}(T_{jk}, d \mu_{R}dt)}
& \lesssim&
\|  \Psi_{\tau} \|_{L^{p}L^{q}(T_{jk}, d \mu_{R}dt)}\\
&  \leq &
\mu_{R}(I_{j})^{\frac{1}{p} - \frac{1}{q}} 
\| \Psi_{\tau} \|_{L^{q}(T_{jk}, d \mu_{R} d t)} \\ \nonumber
& \lesssim & 
 \mu_{R}(I_{j})^{\frac{1}{p} - \frac{1}{q}} 
c_{\alpha}(\mu)^{\frac{1}{q}}
\| \Psi_{\tau} \|_{L^{q}(T_{jk}, dx dt)} \\ \label{ApplPsiProp}
& \lesssim &
c_{\alpha}(\mu)^{\frac{1}{q}}  
 \mu_{R}(I_{j})^{\frac{1}{p} - \frac{1}{q}}  
 R^{\frac{n+2}{2q}\varepsilon}
 |T_{jk}|^{\frac{1}{q}},
 %% che poi e' \\ 
%& = & R^{\varepsilon} c_{\alpha}(\mu ')^{\frac{1}{q}} \delta ^{-\frac{\alpha + 1}{n+1} %\frac{n}{2} - \frac{n-1}{2n} +
%\frac{\alpha - n}{q}}.
\end{eqnarray*}
where the final inequality is by (\ref{PsiPropertyFINAL}). 
Using this and the fact that  $G_{\tau}$ is essentially constant on $T_{jk}$,
\begin{eqnarray}\nonumber \| \Psi_{\tau} G_{\tau} \|_{L^{p}L^{\infty}(T_{jk}, d \mu_{R}dt)}
 & \lesssim & 
 c_{\alpha}(\mu)^{\frac{1}{q}} 
\mu_{R}(I_{j})^{\frac{1}{p} - \frac{1}{q}}  R^{\frac{n+2}{2q}\varepsilon}|T_{jk}|^{\frac{1}{q}}
% che e' \delta^{-\frac{\alpha + 1}{n+1} \frac{n}{2} - \frac{n-1}{2n} +
%\frac{\alpha - n}{q}} 
 \| G_{\tau} \|_{L^{\infty}L^{\infty}(T_{jk})}
\\ \nonumber
 & \lesssim & 
 c_{\alpha}(\mu)^{\frac{1}{q}} 
\mu_{R}(I_{j})^{\frac{1}{p} - \frac{1}{q}} R^{\frac{n+2}{4}\varepsilon} |T_{jk}|^{\frac{1}{q}}
% che e' \delta^{-\frac{\alpha + 1}{n+1} \frac{n}{2} - \frac{n-1}{2n} +
%\frac{\alpha - n}{q}} 
 \delta^{\frac{n+1}{2}}\delta^{\frac{1}{p}} \| G_{\tau} \|_{L^{2}L^{p}(T_{jk})}.
\end{eqnarray} 
Plugging this into \eqref{I}, we obtain 
\begin{align*} 
&\quad \|  \Psi_{\tau}G_{\tau} \|_{L^{2}( d\mu_{R}) L^{\infty}(0,R)} \\  
&\!\! \!\!\lesssim 
 c_{\alpha}(\mu)^{\frac{1}{q}}
\mu_{R} (I_{j})^{\frac{1}{2}- \frac{1}{q}} R^{\frac{n+2}{4}\varepsilon}|T_{jk}|^{\frac{1}{q}} \delta^{\frac{n+1}{2}}\delta^{\frac{1}{p}}
\bigg( \sum_{j} \bigg( \sum_{k} 
\|  G_{\tau} \|^{p}_{L^{2}L^{p}(T_{jk})}
\bigg)^{2/p} \bigg)^{1/2}
\\ 
& \!\!\!\!\lesssim  
\sqrt{c_{\alpha}(\mu)}R^{\frac{n+2}{4}\varepsilon} 
\delta^{-(\alpha+1)(\frac{1}{2}-\frac{1}{q})}\delta^{(n+1)(\frac{1}{2}-\frac{1}{q}) + \frac{1}{p}-\frac{1}{q}} 
\bigg( \sum_{j} \bigg( \sum_{k} 
\|  G_{\tau} \|^{p}_{L^{2}L^{p}(T_{jk})}
\bigg)^{2/p} \bigg)^{1/2}\!\!\!\!,
\end{align*}
where in the second inequality we use $\mu_{R}(I_{j}) \lesssim c_\alpha(\mu) \delta^{- (\alpha +1)}$ which follows by covering the $I_{j}$ by $\delta^{-1}$ balls of radius $\delta^{-1}$. Finally, using  that $G_\tau$ is essentially constant on $T_{jk}$ and $\frac{1}{2}-\frac{1}{q}=\frac{1}{2n}$,  we can sum up to obtain
\begin{equation}\label{ApplDefMeasu}
 \|  \Psi_{\tau}G_{\tau} \|_{L^{2}( d\mu_{R}) L^{\infty}(0,R)}  \lesssim 
\sqrt{c_{\alpha}(\mu)}R^{\frac{n+2}{4}\varepsilon} 
\delta^{\frac{n-\alpha}{2n} + \frac{1}{p}-\frac{1}{q}} 
\|  G_{\tau} \|_{L^{2}(B_{R})L^{p}(0,R)}.
\end{equation}
In fact we have only performed this argument for general $p$ to facilitate the proof of the following lemma. Here we set $p=2$ and so it remains to bound
\begin{eqnarray*}
\| |T_{\tau} g|*\zeta_{\tau'}\|_{{L^{2}(B_{R})L^{2}(0,R)}}&\le & \int \|T_{\tau} g(\cdot-y)\|_{{L^{2}(0,R)L^{2}(\R^n)}}\zeta_{\tau'} (y) \ d y\\
&\le  &\int \|g_\tau\|_{{L^{2}(0,R)L^{2}(\R^n)}}\zeta_{\tau'} (y) \ d y\\
&\lesssim &\| g_\tau\|_{L^{2}(0,R)L^{2}(\R^n)}=R^{1/2}\|g_\tau\|_{L^{2}(\R^n)},
\end{eqnarray*}
by Fubini, Minkowski's integral inequality and Plancherel.
Plugging this into the previous estimate,  we see that
$$
\big\|\Psi_{\tau}|T_{\tau} g|*\zeta_{\tau'}\big\|_{L^{2}( d \mu_{R}) L^{\infty}(0,R)} \lesssim \sqrt{c_{\alpha}(\mu)}R^{\frac{1}{2}+\frac{n+2}{4}\varepsilon} 
\delta^{\frac{n-\alpha}{2n}+\frac{1}{2n}}  \| g_\tau \|_{2},
$$
which, with $R^{-1/2}\le \delta \le R^{-1/2+\varepsilon}$, yields the desired uniform estimate. \end{proof}

\begin{lemma}Let $0<\varepsilon<\frac{1}{8n}$. Then,
for all caps $\tau\sim\delta$ with $R^{-1/2}  \le \delta \leq 1$ and all $(n+1)$-transversal caps $\tau_1,\ldots\tau_{n+1}\sim \delta/K_{n+1}$ contained in~$\tau$, 
\begin{align*}
\Big\| \Psi_{\tau}\prod_{k=1}^{n+1}
|T_{\tau_{k}} g|^{\frac{1}{n+1}}
*&\zeta_{\tau_k'} \Big\|_{L^{2}( d \mu_{R}) L^{\infty}(0,R)}\\ &\lesssim \sqrt{c_{\alpha}(\mu)}R^{n\varepsilon}
\max \left( R^{\frac{1}{2}- \frac{1}{4n} - \frac{n - \alpha}{4n}},
R^{\frac{n}{2(n+1)}} \right)  \| g_\tau \|_{2}.
\end{align*} 
\end{lemma}

\begin{proof} 
As before we set $G_{\tau} :=  \prod_{k=1}^{n+1}  
|T_{\tau_{k}}g|^{\frac{1}{n+1}}  * \zeta_{\tau_{k}'}$, and this time we will prove 
$$
\|  \Psi_{\tau}G_{\tau} \|_{L^{2}( d\mu_{R}) L^{\infty}(0,R)}   
\lesssim
\sqrt{c_{\alpha}(\mu)}R^{\frac{n}{2(n+1)}+n\varepsilon}
\delta^{ \frac{n - \alpha}{2n}+\frac{1}{2n}- \frac{1}{n+1}}
\| g_{\tau} \|_{2}, 
$$
which yields the desired estimate uniform in the range $R^{-1/2} \le \delta \leq 1$.
Covering $B_{R} \times (0,R) $ by  translations of $\tau'$, as $\tau'\subset \tau_k'$ we still have that
$G_{\tau}$ is essentially constant at this scale. Repeating the previous argument, this time with $p := \frac{2(n+1)}{n}$, by \eqref{ApplDefMeasu} we have 
\begin{equation*} 
\|  \Psi_{\tau}G_{\tau} \|_{L^{2}( d\mu_{R}) L^{\infty}(0,R)} \lesssim 
\sqrt{c_\alpha(\mu)}R^{\frac{n+2}{4}\varepsilon} 
\delta^{\frac{n-\alpha}{2n} +\frac{1}{p}- \frac{1}{q}} 
\|  G_{\tau} \|_{L^{2}(B_{R})L^{p}(0,R)},
\end{equation*}
and so it remains to bound
$
\| G_{\tau} \|_{{L^{2}(B_{R})L^{p}(0,R)}}.
$ 
By Minkowski's integral inequality, it will suffice to treat
$$
 \int\Big\|\prod_{k=1}^{n+1} |T_{\tau_{k}} g_{y_{k}}|^{\frac{1}{n+1}}
\Big\|_{L^{2}(B_{R})L^{p}(0,R)} \prod_{k=1}^{n+1}\zeta_{\tau_k'}(y_k)\, d y_{1}\ldots dy_{n+1},
$$
where
\begin{equation}\nonumber
g_{y_{k}} := g\,\chi_{\tau_k} e^{-i \pi(y_{k}) \cdot \xi +t_k| \xi |^{2}},\qquad t_k := y_{k} - \pi(y_{k}).
\end{equation}
Noting that $\frac{1}{p}-\frac{1}{q}=\frac{1}{2n}- \frac{1}{2(n+1)}$ and $\|g_{y_{k}}\|_2=\|g_{\tau_k}\|_2$, it remains to prove 
$$
\Big\|\prod_{k=1}^{n+1} |T_{\tau_{k}} g|^{\frac{1}{n+1}}
\Big\|_{L^{2}(B_{R})L^{p}(0,R)} \lesssim R^{\frac{n}{2(n+1)}+\varepsilon}\delta^{-\frac{1}{2(n+1)}}\|g\|_{2}.
$$
By scaling as in the proof of Lemma~\ref{Fuman} (see \eqref{polk} for the definition), this would follow from 
$$
\Big\|\prod_{k=1}^{n+1} |\widetilde{T}_{\widetilde{\tau}_k}f(x' -2\delta^{-1} \xi_{0}t', t')|^{\frac{1}{n+1}}
\Big\|_{L^{2}(B_{\delta R})L^{p}(0,\delta^2R)} \lesssim R^{\frac{n}{2(n+1)}+\varepsilon}\delta^{\frac{2}{p}-\frac{1}{2(n+1)}}\|f\|_{2}.
$$
By a rotation we can suppose that $\xi_0$ is parallel to $x_n$, so by an application of H\"older's inequality, and making the change of variables $x=x' - 2\delta^{-1}\xi_{0} t'$, it would suffice to prove
\begin{equation*}\label{fgd}
\Big\|\prod_{k=1}^{n+1} |\widetilde{T}_{\widetilde{\tau}_k}f|^{\frac{1}{n+1}}
\Big\|_{L^{2}(B_{\delta R})L_{x_n,t}^{p}(-2\delta R,2\delta R)\times(0,\delta^2R)} \lesssim R^{\frac{n-1}{2(n+1)}+\varepsilon}\delta^{\frac{3}{p}-\frac{1}{2}-\frac{1}{2(n+1)}}\|f\|_{2}.
\end{equation*}
Now partitioning $\R^{n-1}$ into cubes $\Omega$ of side length $\delta^2R$, and applying H\"older's inequality, the left-hand side is bounded by 
$$
(\delta^2R)^{2(n-1)(\frac{1}{2}-\frac{1}{p})}\left(\sum_{\Omega}\Big\|\prod_{k=1}^{n+1} |\widetilde{T}_{\widetilde{\tau}_k}f|^{\frac{1}{n+1}}
\Big\|^2_{L^{p}(\Omega)L_{x_n,t}^{p}(-2\delta R,2\delta R)\times(0,\delta^2R)}\right)^{1/2}.
$$
Noting that $$2(n-1)\Big(\frac{1}{2}-\frac{1}{p}\Big)=\frac{n-1}{n+1}=\frac{3}{p}-\frac{1}{2}-\frac{1}{2(n+1)},$$
the proof is completed by an application of Proposition~\ref{MultilinEstimates2}.
\end{proof}

\section*{Appendix}

The following lemma is well-known; see for example \cite[pp. 1024]{Tem}.

\begin{lemma}\label{TechnicLemma2}
Let $\widehat{\psi}=\widehat{\psi}_{\mathrm{o}}\ast\widehat{\psi}_{\mathrm{o}}$ be   
a smooth radially symmetric cut-off function supported in
$B(0,d) \subset \mathbb{R}^{d}$ and equal to one on $B(0,\sqrt{d})$ and consider the scaled version $\phi_{\tau'}$ adapted to $\tau'$.
Then, for all $m\ge 1$,  
\begin{equation}\label{TechnicLemma2Formula}
| F(x,t) | \lesssim
\left( |F|^{\frac{1}{m}} * | \psi_{\tau'}|^{\frac{1}{m}} (x,t) \right)^{m},
\end{equation}
provided $\supp \widehat{F} \subset \tau\subset \R^{d}$. 
\end{lemma}

\begin{proof}
As usual we set $m' := m/(m-1)$.
Letting
\begin{equation}
\eta(x,t) :=
\psi_{\tau'}(x,t) e^{i x \cdot \xi_{\mathrm{o}} +i t\phi(\xi_{\mathrm{o}})},
\end{equation}
where $(\xi_{\mathrm{o}},\phi(\xi_{\mathrm{o}}))$ is the centre of $\tau$, we note that
\begin{equation}\label{BarEtaVsEta}
|\eta|^{\frac{1}{m}}
=
|\tau'|^{\frac{1}{m'}}|\psi_{\tau'}|^{\frac{1}{m}}.
\end{equation}
By the self reproducing formula $ F = F * \eta $ 

\begin{eqnarray} \label{ByDvision1} 
 |F(x,t)|   
& \leq & 
 \int |F ((x,t)- y) \eta(y)| \ dy 
 \\ \nonumber
& = & \int |F((x, t)- y)\eta(y)|^{\frac{1}{m}}
 | F((x,t)- y) \eta(y)|^{\frac{1}{m'}} \ dy, 
\\ 
\nonumber
& \leq & 
 \| F ((x,t)- \cdot)  \eta\|^{\frac{1}{m'}}_{L^{\infty}} 
\int | F((x,t) - y)  \eta(y) |^{\frac{1}{m}} \ dy 
 \\\nonumber
& \lesssim & 
|\tau'|^{-\frac{1}{m'}}
\| F ((x,t)- \cdot)  
\eta \|^{\frac{1}{m'}}_{L^{1}}  
 \int | F((x,t)- y ) \eta(y)|^{\frac{1}{m}} \ dy ,\nonumber
\end{eqnarray}
where in the last inequality we have used Bernstein's 
inequality.
Hence by dividing  
by $\| F ((x,t)- \cdot) \eta \|^{\frac{1}{m'}}_{L^{1}}$, we see that 
\begin{eqnarray}\nonumber
\left( \int |F ((x,t)- y) \eta(y)| 
\ dy \right)^{\frac{1}{m}} 
 & \lesssim  &
|\tau'|^{-\frac{1}{m'}}
\left( \int |F ((x,t)- y) 
\eta(y)|^{\frac{1}{m}} \ dy \right) 
\\ \label{TechnLemma2FinalFormula}
 & =  &
\left( \int |F ((x,t)- y)|^{\frac{1}{m}} 
|\psi_{\tau'}(y)|^{\frac{1}{m}}\ dy \right) ,
\end{eqnarray}
where in the final identity we have used (\ref{BarEtaVsEta}).
Then (\ref{TechnicLemma2Formula})
follows using (\ref{ByDvision1}).
\end{proof}

\begin{lemma}\label{TechnicLemma1}
Let $0 < \delta \le 1$ and let $K > (K')^2 >1$. Let $\Lambda_1,\Lambda_2 \in SO(d)$ be such that $\Lambda_1\Lambda_2^{-1}$ is a rotation by an angle less than $\delta$.
Then if $F: \mathbb{R}^{d} 
\rightarrow \mathbb{R}_{+}$ is essentially constant on translates of $\Lambda_1^{-1}(T)$ where
$$T := \Big[-\frac{K}{\delta}, \frac{K}{\delta} \Big] \times \dots \times \Big[-\frac{K}{\delta}, \frac{K}{\delta} \Big]
\times\Big[-\frac{K}{\delta^2}, \frac{K}{\delta^2} \Big],$$
and
\begin{equation*}
\zeta (x,t) \lesssim 
\frac{\delta^{d+1}}{ (K')^{d}}
\left(1 +  
\Big| \frac{\delta x}{K'}\Big|^{2}
+
\Big| \frac{\delta^{2} t}{K'}\Big|^{2}
\right)^{-N},
\end{equation*}
or 
\begin{equation*}
\zeta (x,t) \lesssim 
\frac{\delta^{d+1}}{ (K')^{d+1}}
\left(1 +  
\Big| \frac{\delta x}{K'}\Big|^{2}
+
\Big| \frac{\delta^{2} t}{(K')^2}\Big|^{2}
\right)^{-N},
\end{equation*}
for some $N \ge d$, then
\begin{equation*}\label{TechnicLemma1Formula}
F*\zeta(\Lambda_2(\cdot)) (x_1,t_1) \lesssim_{L}
F(x_2,t_2) + \| F\|_{L^{\infty}} 
\left( \frac{K'}{K}\right)^{N}
\end{equation*}
whenever $(x_1,t_1) -(x_2,t_2) \in\Lambda_1^{-1}(T)$.
\end{lemma} 

\begin{proof}
If $\zeta$ takes the second form, then by a change of variables, 
\begin{align*}
\int_{\mathbb{R}^{d} / T} \zeta(x,t) \ dxdt
& \lesssim 
\int_{\mathbb{R}^{d} / [-K/K',K/K']^{d-1}\times[-K/(K')^2,K/(K')^2]}
|(x,t)|^{-2N} \ dxdt \\
& \lesssim  \int_{K/K'}^{\infty} \rho^{-2N + d-1} \ d \rho
\lesssim_N \left( \frac{K'}{K}\right)^{2N-d},
\end{align*}
and the same is true if $\zeta$ takes the first form.
Then note that
\begin{eqnarray}\nonumber
& &
\int F((x_1,t_1) - y)\zeta(\Lambda_2(y)) \ dy 
 = 
\int F((x_1,t_1) - \Lambda_2^{-1} y)\zeta(y) \ dy 
\\ \nonumber
& = &
\int_{T} F((x_1,t_1)-\Lambda_2^{-1} y)\zeta(y) \ dy 
+
\int_{\mathbb{R}^{d} / T} F((x_1,t_1)-\Lambda_2^{-1} y)\zeta(y) \ dy 
=: I + II.
\end{eqnarray}
By trigonometry and the essentially constant assumption, we have
\begin{equation}\nonumber
F((x_1,t_1) - \Lambda_2^{-1} y)\Big|_{y \in \Lambda^{-1}_1T} \lesssim F(x_2,t_2),
\end{equation}
whenever $(x_1,t_1)-(x_2,t_2)\in  \Lambda^{-1}_1T$ so that 
$I \lesssim F(x_2,t_2)$. On the other hand, we have that
\begin{equation*}\label{DetailTechnic2}
II \leq \| F \|_{L^{\infty}} \int_{\mathbb{R}^{d}/T} \zeta(y) \ dy
\lesssim \| F \|_{L^{\infty}}
 \left( \frac{K'}{K}\right)^{2N-d}.
\end{equation*}
from before, and so the desired estimate follows by adding the two bounds.
\end{proof}

\end{document}